\documentclass[pdflatex,sn-mathphys-num]{sn-jnl}

\usepackage{graphicx}%
\usepackage{multirow}%
\usepackage{amsmath,amssymb,amsfonts}%
\usepackage{amsthm}%
\usepackage{mathrsfs}%
\usepackage[title]{appendix}%
\usepackage{xcolor}%
\usepackage{textcomp}%
\usepackage{manyfoot}%
\usepackage{booktabs}%
\usepackage{algpseudocode}%
\usepackage{listings}%
\usepackage{float}
\usepackage{color}
\usepackage{mathtools}
\usepackage{breqn}
\usepackage{indentfirst}
\usepackage{arydshln}
\usepackage{enumerate}
\usepackage{cases}
\usepackage{seqsplit}
\usepackage{subfigure}
\usepackage{caption}
\usepackage{colortbl}
\usepackage{makecell}

\lstdefinestyle{maplestyle}{
    basicstyle=\ttfamily\small,
    backgroundcolor=\color{gray!8},
    frame=single,
    framerule=0.4pt,
    rulecolor=\color{gray!60},
    breaklines=true,
    breakindent=0pt,
    postbreak=\mbox{\textcolor{gray}{$\hookrightarrow$}\space},
    columns=flexible,
    keepspaces=true,
    xleftmargin=1em,
    xrightmargin=1em,
    aboveskip=6pt,
    belowskip=6pt,
}

\lstdefinestyle{mapleoutput}{
    style=maplestyle,
    backgroundcolor=\color{blue!5},
    rulecolor=\color{blue!40},
}

\newcommand{\defword}{\textit}

\theoremstyle{thmstyleone}%
\newtheorem{theorem}{Theorem}
\newtheorem{corollary}{Corollary}

\newtheorem{lemma}{Lemma}

\theoremstyle{thmstyletwo}%
\newtheorem{example}{Example}%
\newtheorem{remark}{Remark}%

\theoremstyle{thmstylethree}%
\newtheorem{definition}{Definition}%

\usepackage{tikz}
\usetikzlibrary{fadings}
\usetikzlibrary{shapes,arrows,chains}
\tikzstyle{dec} = [trapezium, trapezium left angle = 80,trapezium right angle=100,minimum width=3cm,minimum height=0.7cm,text centered,draw=black,fill=green!30]
\tikzstyle{inp} = [rectangle,rounded corners, minimum width=3cm,minimum height=0.7cm,text centered, draw=black,fill=red!30]
\tikzstyle{pro} = [rectangle,minimum width=2cm,minimum height=0.7cm,text centered,draw=black,fill=blue!30]
\tikzstyle{arrow} = [thick,->,>=stealth]
\tikzstyle{arrow1} = [thick,-,>=stealth]
\tikzstyle{point}=[coordinate,on grid]

\begin{document}

\title[Minimal multistable zero-one networks]{The Ubiquity of Three Steady States: Minimal Multistable Zero-One Reaction Networks}


\author*{\fnm{Xiaoxian} \sur{Tang}}
\email{xiaoxian@buaa.edu.cn}

\author{\fnm{Jiandong} \sur{Zhang}}


\affil{\orgdiv{School of Mathematical Sciences}, \orgname{Beihang University}, \orgaddress{\street{Xueyuan Road}, \city{Beijing}, \postcode{100191}, \country{China}}}




\abstract{Biochemical networks with unit stoichiometry arise naturally in receptor--ligand binding and multisite phosphorylation systems. A central question is to identify which networks admit multistability. Because this property can be inherited from smaller subnetworks to larger ones, it is natural to seek the smallest networks with it. In this paper, we completely determine all minimal quadratic zero-one networks that exhibit multistability. Building on recent progress that has narrowed the search to the zero-one networks with $3$ species, $6$ reactions, and dimension $3$, say $(3,6,3)$ family. We develop a computational pipeline to provide a complete characterization of multistability within this class. The primary theoretical contribution is a structural simplification showing that for $(3,m,3)$ quadratic zero-one networks, the Jacobian determinants at consecutive positive steady states have opposite signs, and exactly half of these states are stable. This reduces multistationarity detection and stability verification to a local sign-checking problem, eliminating the need for boundary or asymptotic analysis. Applying this result, we identify $373$ quadratic zero-one networks that exhibit bistability (two stable and one unstable positive steady states). Strikingly, no $(3,6,3)$ quadratic zero-one network admits more than $3$ positive steady states---a sharp bound that falls well below the theoretical BKK bound of $5$ and the B\'ezout bound of $8$. Among the $375$ networks that admit exactly $3$ positive steady states, $373$ are multistable and only $2$ are not. These minimal networks provide essential test cases for understanding how bistability emerges in cell signaling without requiring higher-order molecular collisions.}




\maketitle

\section{Introduction}
\label{sec:introduction}

Biochemical reaction networks arising in cell signalling often exhibit 
switch-like behaviour, memory effects, and cell-fate decisions. A 
fundamental mechanism behind such phenomena is the existence of multiple 
stable positive steady states within the same stoichiometric compatibility 
class, namely, multistability 
\cite{ferrell1998allornone,xiong2003memory, craciun2006understanding, conradi2019multistationarity,tyson2022timekeeping}. 
Identifying which networks admit multistability is a central problem in 
the study of reaction networks.

The multistability problem is difficult in general. A common strategy is 
first to detect multistationarity, and then to determine which of the 
corresponding positive steady states are stable. Existing approaches 
include injectivity-based criteria 
\cite{craciun2005multiple1, banaji2016injectivity}, degree-theoretic 
methods \cite{craciun2008homotopy,enciso2014fixed,conradi2017identifying}, 
and tools from computational real algebraic geometry 
\cite{chen2013semialgebraic,maplesoft2020,xia2016automated}. 
However, determining multistability becomes highly challenging as the number of species, reactions, and parameters grows. An important observation is that multistability can often be inherited from smaller subnetworks under certain structural conditions \cite{banaji2018oscillation,banaji2018inheritance}, which motivates the study of small networks as a foundation for understanding larger systems.  
Zero-one reaction networks, where stoichiometric coefficients 
are restricted to $0$ or $1$, have emerged as a fundamental class, 
both because they model many signalling pathways 
\cite{janiakspens2005kinetic,sible2007mathematical,hilioti2008oscillatory} and because their combinatorial simplicity makes large-scale analysis 
tractable. Within this class, the simplest nontrivial case consists of 
linear networks, where every reactant complex has molecularity 
at most one. Such networks are structurally limited: mass-action kinetics 
reduces to linear systems, which cannot exhibit multistability due to the 
absence of nonlinear feedback. The next natural step is the class of 
quadratic zero-one networks, where reactant complexes have 
molecularity at most two, with at least one reaction having reactant 
molecularity exactly two. This setting is the smallest that permits 
nonlinear interactions while remaining chemically realistic---reactions 
with source molecularity at most two are more natural than those 
requiring higher-order collisions---and computationally tractable, as 
the associated mass-action systems remain quadratic 
\cite{banaji2024oscillations}.

Our work is closely related to two recent directions. First, in
\cite{jiao2025multistability}, the smallest zero-one networks admitting
nondegenerate multistationarity and multistability were investigated. In particular, it was
shown that the smallest zero-one networks admitting multistability have $3$ species,
$6$ reactions, and dimension $3$. This identifies the class of $(3,6,3)$ zero-one
networks as the first place where multistability can occur. However, which specific networks within this minimal class actually exhibit multistability remained entirely open.  
In fact, the reason why previous work failed to identify all multistable $(3,6,3)$ zero-one networks is that the search space is enormous, comprising $1{,}367{,}698$ dynamically nontrivial networks. 
Second, in
\cite{jiao2025equivalence,jiao2026equivalence}, an efficient equivalence-reduction algorithm was developed
for zero-one reaction networks based on steady-state ideals, demonstrating that large-scale classification of such networks is computationally feasible. The present paper builds on and extends these two lines of work. We carry out a complete computational search for multistability among all $(3,6,3)$ quadratic zero-one networks. The main contributions of this paper are as follows.
\begin{enumerate}
    \item We establish a structural simplification (Theorem \ref{thm:nomss}) showing that for $(3,m,3)$ quadratic zero-one networks, the Jacobian determinants at consecutive positive steady states always have opposite signs, and exactly half of these states are stable. Consequently, multistability requires at least $3$ nondegenerate positive steady states, and the stability of these states is determined solely by the sign alternation of the Jacobian determinant. This reduces multistationarity detection and stability verification to a local sign-checking problem, eliminating the need for boundary or asymptotic analysis.

    \item We develop a complete computational pipeline for detecting multistable $(3,6,3)$ quadratic zero-one networks, consisting of three stages: (a) Jacobian-sign screening, which uses Theorem 1 to discard all networks whose Jacobian determinant cannot change sign at positive steady states; (b) symbolic real root classification to identify networks admitting between  $3$ and $5$ nondegenerate positive steady states, where the upper bound $5$ is given by the Bernstein--Kushnirenko--Khovanskii (BKK) mixed volume; and (c) stability verification via Jacobian determinant evaluation at each steady state.

    \item Applying this pipeline, we identify $375$ networks that admit exactly $3$ positive steady states, with no network attaining $4$ or $5$. Of these, $373$ are multistable (two stable and one unstable), while only $2$ admit $3$ steady states without multistability (Theorem \ref{thm:main}).
\end{enumerate}

The proof of Theorem~\ref{thm:nomss} relies on a non-standard elimination strategy that departs from the classical computational-algebraic paradigm of triangular decomposition based on resultants and pseudo-division \cite{chen2013semialgebraic,xia2016automated}. Another well-established approach uses Gr\"obner basis triangularization grounded in the Shape Lemma \cite{gianni1989algebraic},  
with subsequent extensions such as the Rational Univariate Representation (RUR) reducing multivariate algebraic systems to univariate rational function representations \cite{rouillier1999solving}. Our method differs fundamentally from both paradigms: instead of eliminating variables to obtain a polynomial or rational univariate function, we construct a univariate function via the positive isoline property. This function is not a polynomial, or even a rational function; its expression may involve radicals arising from the quadratic formula. Nevertheless, its zeros are in one-to-one correspondence with the positive steady states, and the sign of its derivative at each zero encodes the stability of the corresponding state. This construction is possible precisely because the zero-one quadratic structure guarantees the existence of a positive isoline, which would fail for general polynomial systems. In contrast to degree-theoretic methods that typically require compactification at infinity \cite{craciun2008homotopy,conradi2017identifying}, Theorem~\ref{thm:nomss} achieves a structural simplification: for $(3,m,3)$ quadratic zero-one networks, multistationarity and multistability can be determined entirely by the interior behavior of the mass-action system, without recourse to boundary analysis or asymptotic arguments. This reduces the global dynamics to a local sign-checking problem and provides the theoretical foundation for our computational pipeline.

The complete classification of multistable $(3,6,3)$ quadratic zero-one networks provides a catalogue of minimal biochemical switches. These $373$ inequivalent representatives are the smallest zero-one reaction networks capable of generating bistability, making them fundamental building blocks for cell-fate decisions \cite{ferrell1998allornone,xiong2003memory}, memory effects \cite{tyson2022timekeeping}, and switch-like responses in signaling pathways \cite{bagowski2001bistability,craciun2006understanding}. Because multistability can be inherited from subnetworks to larger systems \cite{banaji2018oscillation,banaji2018inheritance}, these minimal networks offer essential test cases for synthetic biology and for the design of robust biochemical circuits. Moreover, the sharp bound of exactly $3$ positive steady states suggests that the combinatorial constraints of zero-one stoichiometry severely restrict the complexity of dynamical behavior, even in the minimal regime where nonlinearity first becomes possible.

The rest of the paper is organized as follows. In Section~\ref{sec:background}, we
review the basic notions on reaction networks,  nondegeneracy, multistationarity,
and multistability. In Section~\ref{sec:thm}, we establish the
structural simplification for $(3,m,3)$ quadratic zero-one networks (Theorem~\ref{thm:nomss})
and develop the positive isoline framework underlying its proof (for readers primarily interested in the computational methodology and biological applications, the technical details of the proof presented in Section \ref{sec:proof} may be skipped without loss of continuity). In Section~\ref{sec:methods}, we present our
computational framework for detecting multistable networks, combining
Jacobian-sign screening, real root classification, and stability analysis. In
Section~\ref{sec:results}, we report the computational result (Theorem \ref{thm:main}), illustrate
the two exceptional networks that admit $3$ positive steady states without
multistability, and present a preliminary analysis of the sharp bound of $3$
positive steady states. Finally, Section \ref{sec:discussion} summarizes our findings and outlines potential directions for future research, particularly regarding the structural analysis of the identified multistable networks.

\section{Background}
\label{sec:background}
 
A \defword{reaction network} $G$ on species $X_1,\dots,X_s$ consists of $m$
reactions
\[
\mu_{1j}X_1+\cdots+\mu_{sj}X_s \xrightarrow{\kappa_j}
\nu_{1j}X_1+\cdots+\nu_{sj}X_s,\qquad j=1,\dots,m,
\]
with nonnegative integer coefficients $\mu_{ij},\nu_{ij}$ and rate constants
$\kappa_j>0$. We assume that $(\mu_{1j},\dots,\mu_{sj})\neq(\nu_{1j},\dots,\nu_{sj})$ for all $j$. Its \defword{stoichiometric matrix} and \defword{reactant matrix} are respectively 
\({\mathcal N}\)  and \({\mathcal Y}\in\mathbb{Z}^{s\times m}\), with
\({\mathcal N}_{ij}=\nu_{ij}-\mu_{ij}\) and \({\mathcal Y}_{ij}=\mu_{ij}\). The image
\(S=\operatorname{im}{\mathcal N}\) is called the
\defword{stoichiometric subspace}, and
\(d=\operatorname{rank}N\) is the  \defword{dimension}. If $s=d$, then the network is \defword{full-dimensional}. The network is \defword{zero-one} if 
$\mu_{ij},\nu_{ij}\in\{0,1\}$ for all $i,j$. We call it \defword{quadratic} if for all $j\in \{1,\ldots,m\}$, we have  $\sum_{i=1}^s\mu_{ij}\le 2$ and there exists $j^*$ such that $\sum_{i=1}^s\mu_{ij^*}=2$. 
 
Under mass-action kinetics the concentration vector
$x\in\mathbb{R}_{\ge0}^s$ evolves by $\dot{x}=f(\kappa,x):={\mathcal N}\,v(\kappa,x)$,
where $v(\kappa,x):=(\kappa_1x^{\mu_1},\dots,\kappa_mx^{\mu_m})^\top$ and
$x^{\mu_j}:=\prod_{i=1}^s x_i^{\mu_{ij}}$. A \defword{positive steady state} is a
vector $x^*\in\mathbb{R}_{>0}^s$ with $f(\kappa,x^*)=0$, and the
\defword{steady-state system} is $ f \subseteq\mathbb{Q}[\kappa, x]$.
If there exists $\kappa \in {\mathbb R}^m_{>0}$ such that $G$ has at least one positive steady state, then we say $G$ is \defword{consistent}. 
Let
\(
\operatorname{Jac}_f(\kappa,x)
\)
denote the Jacobian matrix of \(f\) with respect to \(x\).
A positive steady state \(x^*\) is \defword{nondegenerate} if
\(
\operatorname{im}
\left(
\operatorname{Jac}_f(\kappa,x^*)|_S
\right)
=
S.
\)
Notice that 
if $s=d$, then \defword{\(x^*\) is nondegenerate} if and only if 
$\operatorname{Jac}_f(\kappa,x^*)$ is full rank. 
A steady state $x^*$ is \defword{exponentially stable}, or simply
\defword{stable}, if it is nondegenerate and all nonzero
eigenvalues of \(\operatorname{Jac}_f(\kappa,x^*)\) have negative real parts
\cite{jiao2025multistability}. If a network $G$ has at least one nondegenerate steady state for certain $\kappa\in {\mathbb R}^m_{>0}$, then we say \defword{$G$ is nondegenerate}.  
 
Two networks \defword{have the same form} if one is obtained from the other by
relabeling species and reactions, and they are \defword{equivalent} if, after
such relabeling, they generate the same steady-state ideal generated by $f$ in ${\mathbb Q}(\kappa)[x]$ 
\cite{jiao2025equivalence}. Removing equivalent networks is an important 
preprocessing step in our search, as it greatly reduces the number of systems
that must be analyzed symbolically.
 
A network admits \defword{multistationarity} if there exists
$\kappa\in {\mathbb R}^{m}_{>0}$ giving at least two positive steady states in the same
stoichiometric compatibility class, and \defword{multistability} if at least
two of them are stable. For any full-dimensional network  
 (for instance, $(3,m,3)$ networks studied here), the network has no
conservation laws and the unique compatibility class is the whole positive
orthant. Multistationarity (Multistability) is therefore simply the existence of at least two (stable)
positive steady states of $f=0$.

\section{Theory}\label{sec:thm}

Theorem \ref{thm:nomss} is the main theoretical result of the paper. It establishes that for a $(3,m,3)$ quadratic zero-one network with $N \geq 2$ nondegenerate positive steady states, the Jacobian determinants at consecutive steady states (when ordered by some coordinate) always have opposite signs, which implies that stable and unstable steady states must alternate along that coordinate direction. This alternating sign property, together with the explicit formula for counting stable states given in \eqref{eq:formula}, provides the theoretical foundation for the computational method developed in the next section.

\begin{theorem}\label{thm:nomss}
If a $(3, m, 3)$ quadratic zero-one network has $N$ ($N\geq 2$) nondegenerate positive steady states for a given $\kappa \in \mathbb{R}^m_{>0}$, then the following hold:

\begin{enumerate}[(1)]
    \item There exists an ordering of these steady states as $x^{(1)}, \ldots, x^{(N)}$ with respect to a common coordinate, meaning that there is some $k\in \{1, 2, 3\}$ such that $x_k^{(1)}<\ldots<x_k^{(N)}$, and for each $\ell\in \{1, \ldots, N-1\}$, 
    \[
    \det\!\left(\operatorname{Jac}_f(\kappa, x^{(\ell)})\right)\det\!\left(\operatorname{Jac}_f(\kappa, x^{(\ell+1)})\right)<0.
    \]
    \item The number of stable positive steady states is 
    \begin{align}\label{eq:formula}
    \begin{cases}
    \frac{N}{2}, & \text{if } N \text{ is even},\\
    \frac{N+1}{2} \text{ or } \frac{N-1}{2}, & \text{if } N \text{ is odd}.
    \end{cases}
    \end{align}
\end{enumerate}
\end{theorem}

\begin{remark}\label{rmk:bezout}
Theorem \ref{thm:nomss} implies that a multistable $(3,m,3)$ quadratic zero-one network admits at least $3$ nondegenerate positive steady states. 
By B\'ezout's bound, a $(3,6,3)$ quadratic zero-one network admits at most $8$ (positive) steady states. A sharper bound follows from the BKK mixed volume \cite{bernstein1975numberofroots}: since each steady-state polynomial $f_i$ is a $\mathbb{Q}(\kappa)$-linear combination of monomials from $\mathcal{M}=\{1,x_1,x_2,x_3,x_1x_2,x_1x_3,x_2x_3\}$, the Newton polytope $P$ (the convex hull of $\mathcal{M}$ in $\{0,1\}^3$) has volume $5/6$, giving a BKK bound of $3!\cdot\operatorname{Vol}(P)=5$ (positive) steady states.
\end{remark}

The proof details of Theorem \ref{thm:nomss} are presented in Section \ref{sec:proof}. Before diving into the technical proof, we illustrate the main ideas and framework through a concrete example.

\begin{example}[Cell-fate decision network]\label{ex:pip}  Consider the $(3,6,3)$  quadratic zero-one network: 
\[
\begin{array}{@{}l@{}}
0 \xrightarrow{\kappa_1} \mathbf{P} \xrightarrow{\kappa_2} \mathbf{D}+\mathbf{P}+\mathbf{C},
\\[4pt]
\mathbf{P}+\mathbf{C} \xrightarrow{\kappa_3} \mathbf{D} \xrightarrow{\kappa_4} \mathbf{D}+\mathbf{C},
\\[4pt]
\mathbf{D}+\mathbf{C} \xrightarrow{\kappa_5} 0,\quad \mathbf{D}+\mathbf{P} \xrightarrow{\kappa_6} 0.
\end{array}
\]
This network admits a gene-regulatory interpretation: 
$\mathbf{P}$ (pluripotency factor), $\mathbf{D}$ (differentiation factor), 
and $\mathbf{C}$ (cofactor). The self-activation of $\mathbf{P}$ (reaction~2) 
and the mutual antagonism between $\mathbf{P}$ and $\mathbf{D}$ (reactions~5--6) 
form a minimal switch motif, with $\mathbf{C}$ mediating the 
$\mathbf{P}\to\mathbf{D}$ transition (reactions~3--4).

Denote the concentrations of $\mathbf{D}$, $\mathbf{P}$ and $\mathbf{C}$  by $x_1, x_2$ and $x_3$. The steady-state system is
\[
\begin{aligned}
f_1={}&
 \kappa_2x_2-\kappa_6x_1x_2-\kappa_5x_1x_3
 +\kappa_3x_2x_3,\\
f_2={}&
 \kappa_1-\kappa_6x_1x_2-\kappa_3x_2x_3,\\
f_3={}&
 \kappa_4x_1+\kappa_2x_2-\kappa_5x_1x_3
 -\kappa_3x_2x_3.
\end{aligned}
\]
By eliminating $x_1$ from $f_1$ and $f_2$, we obtain the polynomial $g_{1\to 2}$, which is quadratic in $x_2$ (with coefficients in $\mathbb{Q}[\kappa, x_3]$): 
\[
\begin{aligned}
g_{1\to 2} &= (2\kappa_3\kappa_6x_3 + \kappa_2\kappa_6 )x_2^2 + (\kappa_3\kappa_5x_3^2 - \kappa_1\kappa_6) x_2 - \kappa_1\kappa_5 x_3. 
\end{aligned}
\]
Here, the coefficients of $x_2^2$ and the constant term are denoted by $C_2^{1\to 2}$ and $C_0^{1\to 2}$, respectively. Note that
\[
C^{1\to 2}_0C^{1\to 2}_2= - \kappa_1\kappa_5 x_3(2\kappa_3\kappa_6x_3 + \kappa_2\kappa_6 )<0
\]
for any $(\kappa, x_3)\in {\mathbb R}_{>0}^7$.
So, we can always  get a unique nondegenerate positive solution for $x_2$ from $g_{1\to 2}=0$.
Denote by $x_2=\sigma(\kappa, x_3)$ the positive solution. Similarly, by eliminating $x_1$ from $f_1$ and $f_3$, we obtain
\[
g_{1\to 3} = 2\kappa_3\kappa_5 x_2 x_3^2 + (\kappa_3\kappa_6 x_2^2 - \kappa_3\kappa_4 x_2) x_3 - \kappa_2\kappa_6 x_2^2 - \kappa_2\kappa_4 x_2.
\]
At this point, substituting $x_2=\sigma(\kappa, x_3)$ into $g_{1\to 3}$, we denote the resulting function by $h(\kappa, x_3)$.

It is not difficult to see that for every $\kappa\in {\mathbb R}^6_{>0}$, the positive real solutions of $h=0$ with respect to $x_3$ are in one-to-one correspondence with the positive steady states of the original steady-state system $f$. Moreover, by the construction of $h$ and the structure of zero-one networks, one can further show that the sign of the derivative of $h$ at each positive solution differs from the sign of $\det(\mathrm{Jac}_f)$ at the corresponding positive steady state by exactly one negative sign. Therefore, the conclusion of Theorem~\ref{thm:nomss} follows from the classical fact that the derivative of a univariate ${\mathcal C}^1$-function alternates in sign at consecutive nondegenerate roots.
\end{example}

In Section \ref{sec:proof},  we first present three fundamental lemmas: Lemma \ref{lem:sign-alternation} concerns the alternating signs of derivatives at consecutive roots of a univariate function, Lemma \ref{lm:01} describes the sign pattern of coefficients in the steady-state system of a zero-one network, and Lemma \ref{lm:stab} gives a necessary and sufficient determinant criterion for the stability of positive steady states in $(3,m,3)$ networks.

Next, we establish an elimination framework for $(3,m,3)$ zero-one networks analogous to that in Example \ref{ex:pip}: the ternary system is first reduced to a binary system, and by Lemma \ref{lm:01} the coefficients of the resulting polynomial $g_{i\to j}$ along this elimination path exhibit specific sign patterns (see Lemmas \ref{lm:csign} and \ref{lm:gsys}). Furthermore, when $g_{i\to j}=0$ can always determine (or determine for each parameter value $\kappa$) a unique positive solution for $x_j$ as in Example 1, we say that this elimination path has the (local) positive isoline property (see Definition \ref{def:positive-isoline}). Lemma \ref{lm:hsys} shows that the local positive isoline property is sufficient to derive the conclusion of Theorem \ref{thm:nomss}. Lemma \ref{lm:c0c2nonzero} proves that a general enough $(3,m,3)$ zero-one network always has the positive isoline property. Lemma \ref{lm:c0c2zero} then shows that when certain boundary cases occur, a $(3,m,3)$ quadratic zero-one network can still guarantee the local positive isoline property, thereby completing the entire proof. Note that all lemmas except Lemma \ref{lm:c0c2zero} do not require the quadratic condition. In fact, if the network is not quadratic, the local positive isoline property may fail; for the purpose of focusing on the search for minimal multistable zero-one networks, we omit the detailed discussion of non-quadratic networks.

\subsection{Proof of Theorem \ref{thm:nomss}}\label{sec:proof}


\begin{lemma}\label{lem:sign-alternation}
Let $H(x)\in {\mathcal C}^1(\mathbb{R})$ have $N$ distinct nondegenerate real zeros 
$x_1<x_2<\dots<x_N$ in an open interval ${\mathcal I}\subset \mathbb{R}$, i.e.\ $H(x_i)=0$ and $H'(x_i)\neq 0$ for all $i$. 
Then the derivatives at consecutive zeros have opposite signs:
\[
H'(x_i)H'(x_{i+1})<0,\qquad i=1,\dots,N-1.
\]
\end{lemma}
\begin{proof}
Fix $i\in\{1,\dots,N-1\}$. Since $H$ is continuous and has no zeros in the open interval $(x_i,x_{i+1})$, the function $H$ does not change sign on $(x_i,x_{i+1})$.

Suppose $H>0$ on $(x_i,x_{i+1})$. Then $H$ increases through $x_i$ and decreases through $x_{i+1}$, whence $H'(x_i)>0$ and $H'(x_{i+1})<0$. The case $H<0$ is analogous. In either case $H'(x_i)H'(x_{i+1})<0$.
\end{proof}

\begin{lemma}\label{lm:01}
Given an $(s,m,d)$ zero-one network $G$, let $f$ be its steady-state system.
\begin{enumerate}[(1)]
    \item For every $i\in\{1,\ldots,s\}$, the coefficient of any term in $f_i$ is negative if the term contains $x_i$, and positive otherwise.
    \item The trace of $\mathrm{Jac}_f(\kappa,x)$ is a nonzero polynomial with all negative terms.
    \item If $G$ is full-dimensional ($s=d$) and consistent, then for every $i\in\{1,\ldots,s\}$, $\frac{\partial f_i}{\partial x_i}$ is a nonzero polynomial with all negative terms.
    \end{enumerate}
\end{lemma}
\begin{proof}
{\it (1)} Notice that every term of 
$f_i$ has the form
\[(\nu_{ij}-\mu_{ij})\kappa_jx^{\mu_j}=(\nu_{ij}-\mu_{ij})\kappa_j\prod_{k=1}^s x_k^{\mu_{kj}},\;\;\; \text{where}\; 
\nu_{ij}, \mu_{ij}\in \{0, 1\}.\] If the term contains 
the variable $x_i$, then 
$\mu_{ij}=1$. Since the term indeed appears in $f_i$, $\nu_{ij}-\mu_{ij}\neq 0$. So, 
we have $\nu_{ij}=0$, and hence the coefficient of the term is $-1$. Similarly, 
if the term  contains no 
$x_i$, then we have 
$\nu_{ij}=1$ and $\mu_{ij}=0$, and hence the coefficient is $1$. 

{\it (2)} Since $G$ is zero-one, 
for any $i\in \{1, \ldots, s\}$,
we can write $f_i=p_ix_i+q_i$, where
$p_i$ and $q_i$ are polynomials in 
${\mathbb Q}[\kappa, x]$
containing no $x_i$. 
By {\it (1)}, $p_i$ must be  
a nonzero polynomial with all negative terms or the zero polynomial. 
So,  the trace of $\operatorname{Jac}_f(\kappa,x)$, i.e., $\sum^s_{i=1}\frac{\partial f_i}{\partial x_i}=\sum^s_{i=1}p_i$, is the zero polynomial or 
a nonzero polynomial with all negative terms.  
If the trace is the zero polynomial, then we must have 
$p_i$ is the zero polynomial for every $i$. 
That means in every reaction, 
every species $X_i$ either appears in both sides or does not appear at all. So, the reactant and the product of every reaction are the same, which is impossible by the definition of reaction network.

{\it (3)} Again, we write $f_i=p_ix_i+q_i$.  If $\frac{\partial f_i}{\partial x_i}=p_i$ is the zero polynomial, then $f_i=q_i$. By {\it (1)}, $q_i$ must be  a nonzero polynomial with all positive terms or the zero polynomial. That means either $f_i$ has no positive solutions for any rate constants, or 
$f_i$ is the zero polynomial, which contradicts to the hypothesis that $G$ is  consistent and full-dimensional. So, $\frac{\partial f_i}{\partial x_i}$ must be a nonzero polynomial with all negative terms. 
\end{proof}

\begin{lemma}
\label{lm:stab}
Given a
$(3, m, 3)$
zero-one network \(G\),  let \(f\) be its steady-state system. For any
\(\kappa\in\mathbb{R}_{>0}^{m}\) and any corresponding nondegenerate positive
steady state \(x\in\mathbb{R}_{>0}^{3}\),  \(x\) is
stable if and only if
\(\det\!\left(\operatorname{Jac}_f(\kappa,x)\right)<0.
\)
\end{lemma}
\begin{proof}
Assume that the characteristic polynomial of $\operatorname{Jac}_f(\kappa,x)$ is $$
\lambda^3+b_1\lambda^{2}+b_2\lambda+b_{3}.$$
The Hurwitz matrix ${\mathcal H}$ is equal to
$$\begin{pmatrix}
       b_1&1&0\\
       b_3&b_2&b_1\\
         0&0&b_3
     \end{pmatrix}.$$
By  the Routh-Hurwitz criterion \cite[Criterion 1]{TorresFeliu2021},  all non-zero eigenvalues of $\operatorname{Jac}_f(\kappa,x)$ have negative real parts if and only if 
${\mathcal H}_i>0$ for all $i\in \{1, 2, 3\}$, where
${\mathcal H}_1:=b_1$, ${\mathcal H}_2:=b_1b_2-b_3$ and ${\mathcal H}_3:=b_3(b_1b_2-b_3)$.
Notice that 
$b_1=-\sum^3_{i=1}\frac{\partial f_i}{\partial x_i}(\kappa,x)$.
By Lemma \ref{lm:01} {\it (2)}, 
we know that 
${\mathcal H}_1=b_1>0$.
Since $x$ is nondegenerate (i.e., $\det\!\left(\operatorname{Jac}_f(\kappa,x)\right)\neq 0$),  by \cite[Lemma 6.6 (ii)]{tang2023hopf}, we have ${\mathcal H}_2>0$.
Note that ${\mathcal H}_3=b_{3}{\mathcal H}_2$  and $b_{3}=-\det\!\left(\operatorname{Jac}_f(\kappa,x)\right)$. 
So, 
\(x\) is
stable
if and only if $b_3>0$
(i.e., $\det\!\left(\operatorname{Jac}_f(\kappa,x)\right)<0)$. 
\end{proof}

Given a $(3,m,3)$ zero-one network, let $f$ be its steady-state system. For each $\ell \in \{1, 2, 3\}$, we write
\begin{align}\label{eq:fi}
f_{\ell} = a_0^{(\ell)} + a_1^{(\ell)}x_1 + a_2^{(\ell)}x_2 + a_3^{(\ell)}x_3
     + a_{1,2}^{(\ell)}x_1x_2 + a_{1,3}^{(\ell)}x_1x_3 + a_{2,3}^{(\ell)}x_2x_3+a_{1,2,3}^{(\ell)}x_1x_2x_3, 
    \end{align}
where each $a^{\ell}_{*}$ is a linear combination of $\kappa_1, \ldots, \kappa_m$. By Lemma \ref{lm:01} {\it (1)}, we know that $a^{\ell}_{*}$ can be either written as 
$\kappa_{i_1}+\ldots+\kappa_{i_n}$ or $-(\kappa_{i_1}+\ldots+\kappa_{i_n})$. If the corresponding monomial does not appear in $f_i$, then we say $a^{\ell}_{*}\equiv 0$ or $a^{\ell}_{*}$ is the zero polynomial. 

Let $\mathfrak{S}_3$ denote the set of all the permutations of $\{1, 2, 3\}$.
For any $(i, j, k)\in \mathfrak{S}_3$, we define 
\begin{align}
g_{i\to j}:=& \frac{\partial f_i}{\partial x_i} f_j-\frac{\partial f_j}{\partial x_i} f_i. \label{eq:ggeneral}
\end{align}
The meaning of $g_{i\to j}$ is to eliminate $x_i$ in $f_j$
using $f_i$. For readers familiar with computational algebra, this notation denotes the resultant of $f_i$ and $f_j$ with respect to the variable $x_j$. Notice that $g_{i\to j}\in {\mathbb Q}[\kappa, x_j, x_k]$, and the degree of $g_{i \to j}$ with respect to $x_j$ is at most $2$. So, we can write 
\begin{align}\label{eq:gij}
g_{i \to j}=C^{i\to j}_2 x^2_j + C^{i\to j}_1 x_j +C^{i\to j}_0,
\end{align}
where
\begin{align}
C^{i\to j}_2&=-(a^{(j)}_{i,j}+a^{(j)}_{i,j,k}x_k)(a^{(i)}_j+a^{(i)}_{j,k}x_k)+(a^{(i)}_{i,j}+a^{(i)}_{i,j,k}x_k)(a^{(j)}_j+a^{(j)}_{j,k}x_k), \label{eq:c2}\\
C^{i\to j}_0&=-(a^{(j)}_i+a^{(j)}_{i,k})(a^{(i)}_0+a^{(i)}_{k}x_k)+(a^{(i)}_i+a^{(i)}_{i,k})(a^{(j)}_0+a^{(j)}_{k}x_k).  \label{eq:c0}\\
C^{i\to j}_1 &= (a_{i,j}^{(i)}+a_{i,j,k}^{(i)}x_k)\bigl(a_0^{(j)} + a_k^{(j)}x_k\bigr) 
     - (a_{i,j}^{(j)}+a_{i,j,k}^{(j)}x_k)\bigl(a_0^{(i)} + a_k^{(i)}x_k\bigr)\notag \\
    &\quad + \bigl(a_i^{(i)} + a_{i,k}^{(i)}x_k\bigr)
            \bigl(a_j^{(j)} + a_{j,k}^{(j)}x_k\bigr)
     - \bigl(a_i^{(j)} + a_{i,k}^{(j)}x_k\bigr)
       \bigl(a_j^{(i)} + a_{j,k}^{(i)}x_k\bigr).\label{eq:c1}
\end{align}
\begin{remark}\label{rmk:order}
By the setting presented in \eqref{eq:fi}, we only have the notation $a^{(\ell)}_{i,j}$ for $i<j$ (e.g., $a_{1,2}^{(\ell)}$ for $x_1x_2$, not $a_{2,1}^{(\ell)}$). Since $i$ and $j$ in \eqref{eq:c2}--\eqref{eq:c1} come from an arbitrary permutation $(i,j,k)\in\mathfrak{S}_3$ where $i>j$ may occur, we adopt the convention that $a_{i,j}^{(\ell)}:=a_{j,i}^{(\ell)}$ whenever $i>j$.
\end{remark}
\begin{lemma}\label{lm:csign}
 In ${\mathbb Q}[\kappa, x_k]$,
$C^{i\to j}_2$ is  
a polynomial with all positive terms or the zero polynomial, and $C^{i\to j}_0$ is  
a polynomial with all negative terms or the zero polynomial.
\end{lemma}
\begin{proof}
Notice that if a monomial does not appear in $f_i$, then the corresponding coefficient $a^{(\ell)}_{*}$ defined in \eqref{eq:fi} is identically zero. Below, we abuse the notation ``$a^{(\ell)}_{*} \leq 0$ (or $\geq 0$)" a little to denote  that $a^{(\ell)}_{*}$ is a polynomial with all negative (positive) terms or the zero polynomial.  
  By Lemma \ref{lm:01} {\it (1)}, we have 
\begin{align*}
a^{(i)}_{i},\;\; a^{(i)}_{i,j},\;\; a^{(i)}_{i,k} ,\;\;a^{(i)}_{i,j,k} &\leq 0,\\
a^{(i)}_0, \;\; a^{(i)}_j,\;\; a^{(i)}_k, \;\; a^{(i)}_{j,k} &\geq 0.
\end{align*}  
By symmetry, we have
\begin{align*}
a^{(j)}_{j},\;\; a^{(j)}_{i,j},\;\; a^{(j)}_{j,k},\;\;a^{(j)}_{i,j,k} &\leq 0,\\
a^{(j)}_0, \;\; a^{(j)}_i,\;\; a^{(j)}_k, \;\; a^{(j)}_{i,k} &\geq 0.
\end{align*}  
So, the conclusion follows from 
\eqref{eq:c2} and \eqref{eq:c0}.
\end{proof}

We denote by ${\mathcal K}(G)$ the subset of ${\mathbb R}_{>0}^m$ such that for any $\kappa \in {\mathcal K}$, the network $G$ has at least one positive steady state. When the context is clear, we simply write ${\mathcal K}(G)$ as ${\mathcal K}$. Notice that $G$ is consistent if and only if ${\mathcal K}(G)\neq \emptyset$. 

\begin{lemma}\label{lm:gsys}
Given a $(3,m,3)$ zero-one network $G$, for any  $(i,j,k) \in \mathfrak{S}_3$,  let  $g_{i\to j}$ and $g_{i\to k}$ be defined as in \eqref{eq:ggeneral}.  
Then, for any $\kappa^*\in {\mathcal K}$,  there is a one-to-one correspondence between  $\{(x_j, x_k)\in {\mathbb R}^2_{>0}|g_{i\to j}(\kappa^*, x_j, x_k)=g_{i\to k}(\kappa^*,x_j,x_k)=0\}$ and
the positive steady states of $G$, and for any positive steady state $x^*$, we have 
\begin{align}\label{eq:jacg}
\det\!\left(\operatorname{Jac}_g(\kappa^*, x_j^*, x_k^*)\right)
&=\frac{\partial f_i}{\partial x_i}(\kappa^*, x^*)\det\!\left(\operatorname{Jac}_f(\kappa^*, x^*)\right).
\end{align}
\end{lemma}
\begin{proof}
Notice that by \eqref{eq:ggeneral},
we have $g_{i\to j}, g_{i\to k}\in {\mathbb Q}[\kappa, x_j, x_k]$, and for any $\kappa^*\in {\mathcal K}$,   if  $x^*=(x^*_i, x^*_j, x^*_k)$ is a positive steady state of $G$, then $(x^*_j, x_k^*)$ must be a solution of $g_{i\to j}=g_{i\to k}=0$. On the other hand, 
suppose $(x^*_j, x_k^*)$ is a positive solution of $g_{i\to j}=g_{i\to k}=0$. Let  $x^*_i=-\frac{q_i(\kappa^*, x^*_j, x^*_k)}{p_i(\kappa^*,  x^*_j, x^*_k)}$, where
we assume that $f_i$ can be written as $f_i=p_ix_i+q_i$. Recall that by Lemma \ref{lm:01} {\it (1)} and {\it (3)}, we have $x^*_i>0$. Then, $x^*:=(x^*_i, x^*_j, x^*_k)$ is a positive solution to $f_i=0$. Notice again that by Lemma \ref{lm:01} {\it (3)}, we have $\frac{\partial f_i}{\partial x_i}(x^*)=p_i(x^*)<0$.  So, by \eqref{eq:ggeneral}, $x^*$ is also a common positive solution to $f_j=f_k=0$. Therefore, $x^*$ is a positive steady state of $G$.

Notice that   the Jacobian matrix $\operatorname{Jac}_g$ is  \begin{align}\notag
\begin{pmatrix}
\dfrac{\partial g_{i\to j}}{\partial x_j} & \dfrac{\partial g_{i\to j} }{\partial x_k } \\
\dfrac{\partial g_{i\to k}}{\partial x_j} & \dfrac{\partial g_{i \to k}}{\partial x_k}
\end{pmatrix}. 
\end{align}
By \eqref{eq:ggeneral}, for any rate-constant vector $\kappa^*\in {\mathcal K}$ and a corresponding positive steady state 
$x^*$,
we have 
\begin{align*}
\operatorname{Jac}_g(\kappa^*, x_j^*, x_k^*)
&=\begin{pmatrix}
\det\dfrac{\partial(f_i,f_j)}{\partial(x_i,x_j)} & \det\dfrac{\partial(f_i,f_j)}{\partial(x_i,x_k)} \\
\det\dfrac{\partial(f_i,f_k)}{\partial(x_i,x_j)} & \det\dfrac{\partial(f_i,f_k)}{\partial(x_i,x_k)}
\end{pmatrix}|_{(\kappa,x)=(\kappa^*, x^*)}. 
\end{align*}
Hence, we have \eqref{eq:jacg} holds. 
\end{proof}

\begin{corollary}\label{cry:csign}
Given a $(3, m, 3)$ zero-one network $G$, 
if there exists $(i,j,k)\in \mathfrak{S}_3$  such that 
$C^{i\to j}_0C^{i\to j}_2$ is the zero polynomial and $C^{i\to j}_1$ is also the zero polynomial, then the network $G$ is inconsistent or degenerate. 
\end{corollary}
\begin{proof}
If $C^{i\to j}_0C^{i\to j}_2$ is the zero polynomial and $C^{i\to j}_1$ is also the zero polynomial, then by \eqref{eq:gij}, we have $g_{i \to j}=C^{i\to j}_2x^2_k$ or  $g_{i \to j}=C^{i\to j}_0$. 
Thus, the conclusion directly follows from Lemma \ref{lm:csign} and Lemma \ref{lm:gsys}.
\end{proof}

\begin{definition}[Positive isoline property]\label{def:positive-isoline}
Given a $(3,m,3)$   zero-one network $G$, any $(i,j,k)\in \mathfrak{S}_3$ is said to have the \defword{positive isoline property} if 
$g_{i\to j}=0$ defines a ${\mathcal C}^1$-function $\sigma:\mathcal{K}\times \mathbb{R}_{>0}\to\mathbb{R}_{>0}$ such that
for any $(\kappa,x_k)\in\mathcal{K}\times\mathbb{R}_{>0}$, $x_j=\sigma(\kappa,x_k)$ is the unique 
nondegenerate positive solution of $g_{i\to j}(\kappa,x_j,x_k)=0$. We  say
$\sigma:\mathcal{K}\times\mathbb{R}_{>0}\to\mathbb{R}_{>0}$ is the \defword{positive isoline function according to $(i,j,k)$}. 

 Any $(i,j,k)\in \mathfrak{S}_3$ is said to have the \defword{local positive isoline property} if for any $\kappa^*\in \mathcal{K}$, 
$g_{i\to j}|_{\kappa=\kappa^*}=0$ defines a ${\mathcal C}^1$-function $\sigma^*: \mathcal{I}\to\mathbb{R}_{>0}$, 
where $\mathcal{I}\subset\mathbb{R}_{>0}$ is an open interval, such that
for any $x_k\in \mathcal{I}$, $x_j=\sigma^*(x_k)$ is the unique 
nondegenerate positive solution of $g_{i\to j}(\kappa^*,x_j,x_k)=0$, and for any 
$x_k\in \mathbb{R}_{>0}\setminus\mathcal{I}$, the equation 
$g_{i\to j}(\kappa^*,x_j,x_k)=0$ has no positive solutions. We also say
$\sigma^*:\mathcal{I}\to\mathbb{R}_{>0}$ is the \defword{local positive isoline function according to $(i,j,k)$}.
\end{definition}

\begin{remark}
Notice that if $(i,j,k)\in \mathfrak{S}_3$ has the positive isoline property and $\sigma$ is its positive isoline function, then for 
any $\kappa^*\in \mathcal{K}$, $\sigma|_{\kappa=\kappa^*}:\mathbb{R}_{>0}\to\mathbb{R}_{>0}$ is naturally the local positive isoline function, and so, $(i,j,k)$ also has the local positive isoline property. 

By Definition \ref{def:positive-isoline},
in Example \ref{ex:pip}, $(1,2,3)$ has the positive isoline property, and
the function $\sigma:\mathcal{K}\times\mathbb{R}_{>0}\to\mathbb{R}_{>0}$ derived from $g_{1\to 2}=0$ is the positive isoline function according to $(1,2,3)$. 
\end{remark}

\begin{lemma}\label{lm:hsys}
Given a $(3,m,3)$   zero-one network $G$, 
if there exists $(i,j,k)\in \mathfrak{S}_3$
having the local positive isoline property and for any $\kappa^*\in {\mathcal K}$, $\sigma^*: \mathcal{I}\to\mathbb{R}_{>0}$ is the local positive isoline function according to $(i,j,k)$, 
then we have the following conclusions. 
\begin{enumerate}[(1)]
    \item For any $\kappa^* 
\in {\mathcal K}$, there is a one-to-one correspondence between $\{x^*_k\in {\mathcal I}| h(x^*_k)=0\}$ and the  positive steady states of $G$, where $h(x_k):=g_{i\to k}(\kappa^*, \sigma^*(x_k), x_k)$, and  for any positive steady state $x^*$, we have 
\begin{align}\label{eq:hjacf}
  {\rm sign}\;h'(x_k^*)
=-{\rm sign}\det\!\left(\operatorname{Jac}_f(\kappa^*, x^*)\right).
\end{align}
\item If  for $\kappa^*\in {\mathcal K}$, $G$ has exactly $N$ $(N\geq 2)$ nondegenerate positive steady states $x^{(1)}, \ldots, x^{(N)}$, where  these steady states are ordered according to their $k$-th coordinates (i.e., $x_k^{(1)}<\ldots<x_k^{(N)}$),  then
for each $\ell \in \{1, \ldots, N-1\}$,  we have 
\begin{align*}
\det\!\left(\operatorname{Jac}_f(\kappa^*, x^{(\ell)})\right)\det\!\left(\operatorname{Jac}_f(\kappa^*, x^{(\ell+1)})\right)<0.
\end{align*}
\item 
If for $\kappa^*\in {\mathcal K}$, $G$ has exactly $N$ ($N\geq 2$) nondegenerate positive steady states, then  
the number of stable positive steady states is 
\begin{align*}
\begin{cases}
\frac{N}{2}, & N\; \text{is even}\\
\frac{N+1}{2}\;
\text{or}\;\frac{N-1}{2},&  N\; \text{is odd}
\end{cases}.
\end{align*}
\end{enumerate}
\end{lemma}
\begin{proof}
{\it (1)} For any $\kappa^*\in {\mathcal K}$, by the definition of $h$, we know that for any positive solution $(x^*_j, x^*_k)$ of $g_{i\to j}=g_{i\to k}=0$, $x^*_k$ must be a solution of $h(x_k)=0$ in ${\mathcal I}$. On the other hand, 
suppose $x^*_k$ is a positive solution of $h(x_k)=0$ in ${\mathcal I}$. Let $x^*_j=\sigma^*(x^*_k)$. Then, $(x^*_j, x^*_k)$ is a positive solution of $g_{i\to j}=g_{i\to k}=0$.  So, by Lemma 
\ref{lm:gsys}, there is a one-to-one correspondence of $\{x^*_k\in {\mathcal I}| h(x_k^*)=0\}$ and the  positive steady states of $G$. 

Notice that $h(x_k):=g_{i\to k}(\kappa^*, \sigma^*(x_k), x_k)$.
Then, for any steady-state point $(\kappa^*, x^*)$, 
\begin{align}\label{eq:jach}
  h'(x_k^*)
=\frac{\det\!\left(\operatorname{Jac}_g(\kappa^*, x_j^*, x_k^*)\right)}{\frac{\partial g_{i\to j}}{\partial x_j}(\kappa^*, x_j^*, x_k^*)}
\end{align}
Notice that by \eqref{eq:ggeneral},
\begin{align}\label{eq:particalgijxj}
\frac{\partial g_{i\to j}}{\partial x_j} = \det\dfrac{\partial(f_i,f_j)}{\partial(x_i,x_j)}.
\end{align}
So, by  \eqref{eq:jacg} and \eqref{eq:jach}, we have 
\begin{align}\label{eq:nice}
  h'(x_k^*)
=\frac{\frac{\partial f_i}{\partial x_i}\det\!\left(\operatorname{Jac}_f\right)}{\det\dfrac{\partial(f_i,f_j)}{\partial(x_i,x_j)}}(\kappa^*, x^*).
\end{align} 
Since the network is zero-one,
by \cite[Lemma 6.2 (i)]{tang2023hopf},
${\det\dfrac{\partial(f_i,f_j)}{\partial(x_i,x_j)}}(\kappa^*, x^*)\geq 0$. 
However, by \eqref{eq:particalgijxj}, 
if ${\det\dfrac{\partial(f_i,f_j)}{\partial(x_i,x_j)}}(\kappa^*, x^*)=0$, then 
$\frac{\partial g_{i\to j}}{\partial x_j}(\kappa^*, x_j^*, x_k^*)=0$, which contradicts the fact that  $\sigma^*(x_k^*)$ is the unique nondegenerate positive solution of $g_2(\kappa^*, x_j, x^*_k)=0$ and $x_j^*=\sigma^*(x_k^*)$. So, we have ${\det\dfrac{\partial(f_i,f_j)}{\partial(x_i,x_j)}}(\kappa^*, x^*)>0$. Also, by Lemma \ref{lm:01} {\it (3)}, we have 
$\frac{\partial f_i}{\partial x_i}(\kappa^*, x^*)<0$. Thus, by \eqref{eq:nice}, we have \eqref{eq:hjacf}.

{\it (2)} If for $\kappa^*\in {\mathcal K}$,  $G$ has exactly $N$ nondegenerate positive steady states $x^{(1)}, \ldots, x^{(N)}$, then by {\it (1)}, 
the set   
$\{x^*_k\in {\mathcal I}| h(x_k^*)=0\}$ contains exactly 
$\{x_k^{(1)},\ldots, x_k^{(N)}\}$. 
Notice that 
by \eqref{eq:hjacf}, we $ h'(x_k^{(\ell)})\neq 0$. 
Then, by Lemma \ref{lem:sign-alternation},  
for $\ell\in \{1, \ldots, N-1\}$,  we have $h'(x_k^{(\ell)})h'(x_k^{(\ell+1)})<0$. Thus, the conclusion follows from 
\eqref{eq:hjacf}.

{\it (3)}  The conclusion follows from 
 {\it (2)} and Lemma \ref{lm:stab}. 
\end{proof}

\begin{lemma}\label{lm:c0c2nonzero}
 Given a $(3, m, 3)$  zero-one network $G$, 
if there exists  $(i,j,k)\in \mathfrak{S}_3$, such that 
$C^{i\to j}_0C^{i\to j}_2$ is not the zero polynomial, then $(i,j,k)$ has the positive isoline property.
\end{lemma}
\begin{proof}
Recalling \eqref{eq:gij}, we have
$$g_{i\to j}=C^{i\to j}_2x^j_2+C^{i\to j}_1x_j+C^{i\to j}_0,$$
where $C^{i\to j}_{\ell}\in {\mathbb Q}[\kappa, x_k]$. 
By Lemma \ref{lm:csign}, 
if $C^{i\to j}_0C^{i\to j}_2$ is not the zero polynomial, then 
for any $(\kappa, x_k)\in {\mathbb R}^{m+1}_{>0}$, we have $C^{i\to j}_0C^{i\to j}_2<0$. 
Hence, we can always solve $x_j$ from the equation $g_{i\to j}(\kappa, x_j, x_k)=0$ and get two nondegenerate real solutions of opposite signs 
\begin{align}\label{eq:pfun}
   x_j\;=\; \frac{-C^{i\to j}_i\pm \sqrt{(C^{i\to j}_i)^{2}-4C^{i\to j}_0C^{i\to j}_2}}{2C^{i\to j}_2}. 
\end{align}
Denote by $\sigma(\kappa, x_k)$ the positive solution in  \eqref{eq:pfun}, i.e., a positive ${\mathcal C}^1$-function uniquely determined by 
$g_{i\to j}(\kappa, x_j, x_k)=0$. Then, 
$\sigma:\mathcal{K}\times\mathbb{R}_{>0}\to\mathbb{R}_{>0}$ is the positive isoline function according to $(i,j,k)$. 
\end{proof}

\begin{remark}\label{rmk:czero} 
Notice that 
by Lemma \ref{lm:01} {\it (1)} and by \eqref{eq:c2},  
$C^{i\to j}_2$ is  
the zero polynomial if and only if  one of the following  conditions holds.
\begin{align}
&a^{(j)}_j=a^{(j)}_{j,k}=a^{(j)}_{i,j}=a^{(j)}_{i,j,k} \equiv 0\label{eq:c204} \\
&a^{(i)}_j=a^{(i)}_{j,k}=a^{(i)}_{i,j}=a^{(i)}_{i,j,k} \equiv 0\label{eq:c203} \\
&a^{(i)}_j=a^{(i)}_{j,k}=a^{(j)}_j=a^{(j)}_{j,k}\equiv 0 \label{eq:c202} \\
&a^{(i)}_{i,j}=a^{(i)}_{i,j,k}=a^{(j)}_{i,j}=a^{(j)}_{i,j,k}\equiv 0  \label{eq:c201}
\end{align}
By Lemma \ref{lm:01} {\it (1)} and by \eqref{eq:c0}, $C^{i\to j}_0$ is  
 the zero polynomial if and only if one of the following  conditions holds.

 \begin{align}\label{eq:c004}
a^{(i)}_0=a^{(i)}_{k}=a^{(j)}_0=a^{(j)}_{k}\equiv 0
\end{align}

\begin{align}\label{eq:c001}
a^{(i)}_i=a^{(i)}_{i,k}=a^{(j)}_i=a^{(j)}_{i,k}\equiv 0
\end{align} 

\begin{align}\label{eq:c003}
a^{(j)}_i=a^{(j)}_{i,k}=a^{(j)}_0=a^{(j)}_{k}\equiv 0
\end{align}

\begin{align}\label{eq:c002}
a^{(i)}_i=a^{(i)}_{i,k}=a^{(i)}_0=a^{(i)}_{k}\equiv 0
\end{align}

\end{remark}

\begin{lemma}\label{lm:c0c2zero}
Given a $(3,m,3)$ quadratic  zero-one network $G$, 
suppose it is consistent and nondegenerate. 
Then,  we have the following statements. 
\begin{enumerate}[(1)]
\item The condition  \eqref{eq:c204} cannot hold for any $(i,j,k)\in \mathfrak{S}_3$. 
\item If the condition \eqref{eq:c203} holds for some $(i,j,k)\in \mathfrak{S}_3$, then $(i,j,k)$ has the positive isoline property. 
\item The condition  \eqref{eq:c003} cannot hold for any $(i,j,k)\in \mathfrak{S}_3$.
\item If the condition \eqref{eq:c002} holds for some $(i,j,k)\in \mathfrak{S}_3$, then $(i,j,k)$ has the positive isoline property . 

\item If in the conditions \eqref{eq:c202}--\eqref{eq:c001}, only one of  \eqref{eq:c202} and \eqref{eq:c001} holds for some $(i,j,k)\in \mathfrak{S}_3$, then $(i,j,k)$ has  the local positive isoline property. 
\item For any $(\hat{i},\hat{j},\hat{k})\in \mathfrak{S}_3$, 
 if both conditions    \eqref{eq:c202} and \eqref{eq:c001} hold for $(i,j,k)=(\hat{i},\hat{j},\hat{k})$, then
 the conditions    \eqref{eq:c202} and \eqref{eq:c001} cannot hold simultaneously for $(i,j,k)=(\hat{i},\hat{k},\hat{j})$.
 \item If there exists  $(\hat{i},\hat{j},\hat{k})\in \mathfrak{S}_3$ such that both conditions    \eqref{eq:c201} and \eqref{eq:c004} hold for $(i,j,k)=(\hat{i},\hat{j},\hat{k})$ and both
  conditions    \eqref{eq:c202} and \eqref{eq:c001}  hold for $(i,j,k)=(\hat{i},\hat{k},\hat{j})$, then
  the network $G$ admits at most one positive steady state. 
  \item  If there exists  $(\hat{i},\hat{j},\hat{k})\in \mathfrak{S}_3$ such that in the conditions \eqref{eq:c202}--\eqref{eq:c001}, 
 only one of     \eqref{eq:c201} and \eqref{eq:c004} holds for $(i,j,k)=(\hat{i},\hat{j},\hat{k})$ and both
  conditions    \eqref{eq:c202} and \eqref{eq:c001}  hold for $(i,j,k)=(\hat{i},\hat{k},\hat{j})$, then $(\hat{i},\hat{j},\hat{k})$ has the local positive isoline property. 
  \item If for all  $(i, j, k)\in \mathfrak{S}_3$, one of     \eqref{eq:c201} and \eqref{eq:c004} holds, then either   $G$ admits at most one positive steady state, or certain $(\hat{i},\hat{j}, \hat{k})\in \mathfrak{S}_3$ has the local positive isoline property.

   \item If for all  $(i, j, k)\in \mathfrak{S}_3$, one of     \eqref{eq:c202}, \eqref{eq:c201}, \eqref{eq:c004}, \eqref{eq:c001} holds, then either   $G$ admits at most one positive steady state, or certain $(\hat{i},\hat{j}, \hat{k})\in \mathfrak{S}_3$ has the local positive isoline property. 

\end{enumerate}
\end{lemma}
\begin{proof}
Notice that $G$ is quadratic. So, by \eqref{eq:fi},
for each $\ell\in \{1,2,3\}$, we have $a^{(\ell)}_{i,j,k}\equiv 0$ in $f_{\ell}$. 

{\it (1)} If \eqref{eq:c204} holds for some $(i,j,k)\in \mathfrak{S}_3$, then $f_j$ contains no variable $x_j$. By Lemma \ref{lm:01} {\it (1)}, $f_j$ is a polynomial with all positive terms or the zero polynomial, which contradicts the hypothesis that $G$ is consistent and nondegenerate. 

{\it (2)} If \eqref{eq:c203} holds for some $(i,j,k)\in \mathfrak{S}_3$, then by Remark \ref{rmk:czero},  $C^{i\to j}_2\equiv 0$,  
and so, $g_{i\to j}=C^{i\to j}_1x_j+C^{i\to j}_0$.
By \eqref{eq:c1}, we have 
\begin{align}\label{eq:c1in2}
C^{i\to j}_1 = 
     - a_{i,j}^{(j)}\bigl(a_0^{(i)} + a_k^{(i)}x_k\bigr)+ \bigl(a_i^{(i)} + a_{i,k}^{(i)}x_k\bigr)
            \bigl(a_j^{(j)} + a_{j,k}^{(j)}x_k\bigr).
\end{align}
       By Lemma \ref{lm:01} {\it (1)},
       we have $C^{i\to j}_1\geq 0$.  By Corollary \ref {cry:csign}, $C^{i\to j}_1$ cannot be the zero polynomial. By Lemma \ref{lm:csign}, we have $C^{i\to j}_0\leq 0$. Since $G$ is consistent, we know that $C^{i\to j}_0$ cannot be the zero polynomial. 
       So,  
       for any $(\kappa, x_k)\in {\mathbb R}^{m+1}_{>0}$, the equation   $g_{i\to j}(\kappa, x_j, x_k)=0$ always has exactly one nondegenerate positive solution for $x_j$: 
\begin{align}\label{eq:pfuns203}
   x_j\;=\;- \frac{C^{i\to j}_0}{C^{i\to j}_1}. 
\end{align}
Denote by $\sigma(\kappa, x_k)$ the solution in  \eqref{eq:pfuns203}. Then, $(i,j,k)$ has the positive isoline property and $\sigma:\mathcal{K}\times\mathbb{R}_{>0}\to\mathbb{R}_{>0}$ is the corresponding positive isoline function.

{\it (3)} If \eqref{eq:c003} holds for some $(i,j,k)\in \mathfrak{S}_3$, 
then by Remark \ref{rmk:czero}, $C^{i\to j}_0 \equiv0$,  
and so, we have $g_{i\to j}=x_j(C^{i\to j}_2x_j+C^{i\to j}_1)$.
By \eqref{eq:c1},   $C^{i\to j}_1$ is also given by 
\eqref{eq:c1in2}.
 So, by Lemma \ref{lm:01} {\it (1)} and Corollary \ref{cry:csign},
       we have $C^{i\to j}_1>0$. Note that by Lemma \ref{lm:csign}, we have $C^{i\to j}_2\geq 0$.
       Then, we have $g_{i\to j}>0$, and hence, by Lemma \ref{lm:gsys},
       ${\mathcal K}=\emptyset$, which contradicts to that $G$ is consistent. 

{\it (4)} If \eqref{eq:c002} holds for some $(i,j,k)\in \mathfrak{S}_3$, then $C^{i\to j}_0\equiv 0$ 
and $g_{i\to j}=x_j(C^{i\to j}_2x_j+C^{i\to j}_1)$.
By \eqref{eq:c1}, we have 
$$C^{i\to j}_1 = a_{i,j}^{(i)}\bigl(a_0^{(j)} + a_k^{(j)}x_k\bigr)- \bigl(a_i^{(j)} + a_{i,k}^{(j)}x_k\bigr)
       \bigl(a_j^{(i)} + a_{j,k}^{(i)}x_k\bigr).$$
By Lemma \ref{lm:01} {\it (1)},
       we have $C^{i\to j}_1\leq 0$. 
By Corollary \ref {cry:csign}, $C^{i\to j}_1$ cannot be the zero polynomial. By Lemma \ref{lm:csign}, we have $C^{i\to j}_2\geq 0$. Since $G$ is consistent, we know that $C^{i\to j}_2$ cannot be the zero polynomial.  So,  
       for any $(\kappa, x_k)\in {\mathbb R}^{m+1}_{>0}$, the equation   $g_{i\to j}(\kappa, x_j, x_k)=0$ always has exactly one nondegenerate positive solution for $x_j$: 
\begin{align}\label{eq:pfuns002}
   x_j\;=\;- \frac{C^{i\to j}_1}{C^{i\to j}_2}. 
\end{align}
Denote by $\sigma(\kappa, x_k)$ the solution in  \eqref{eq:pfuns002}. Then,  $(i,j,k)$ has the positive isoline property and $\sigma:\mathcal{K}\times\mathbb{R}_{>0}\to\mathbb{R}_{>0}$ is the corresponding positive isoline function.

{\it (5)} Without loss of generality, we assume that $(i,j,k)=(1,2,3)$. By  \eqref{eq:c1}, we can write 
\begin{align*}
C^{1\to 2}_1& = A^{1\to 2}_1 + B^{1\to 2}_1,
\end{align*}
where
\begin{align}\label{eq:A12}
A^{1\to 2}_1& = a_{1,2}^{(1)}\bigl(a_0^{(2)} + a_3^{(2)}x_3\bigr) 
     - a_{1,2}^{(2)}\bigl(a_0^{(1)} + a_3^{(1)}x_3\bigr),
\end{align}
\begin{align*}
B^{1\to 2}_1& = \bigl(a_1^{(1)} + a_{1,3}^{(1)}x_3\bigr)
            \bigl(a_2^{(2)} + a_{2,3}^{(2)}x_3\bigr)
     - \bigl(a_1^{(2)} + a_{1,3}^{(2)}x_3\bigr)
       \bigl(a_2^{(1)} + a_{2,3}^{(1)}x_3\bigr).
\end{align*}
Notice that  for $(i,j,k)=(1,2,3)$,  if one of \eqref{eq:c202} and \eqref{eq:c001} holds, then 
$C^{1\to 2}_1=A^{1\to 2}_1$. 
Suppose that  in \eqref{eq:c202}--\eqref{eq:c001}, only \eqref{eq:c202} holds or only \eqref{eq:c001} holds.  
If only \eqref{eq:c202} holds, then 
$g_{1\to 2}=C^{1\to 2}_1x_2+C^{1\to 2}_0$. 
If only \eqref{eq:c001} holds, then 
$g_{1\to 2}=x_2(C^{1\to 2}_2x_2+C^{1\to 2}_1)$.
By Corollary \ref{cry:csign}, we know that $C^{1\to 2}_1$ cannot be the zero polynomial since $G$ is consistent and nondegenerate. 
So, we can solve $x_2$ from $g_{1\to 2}=0$ and get: 
\begin{align}\label{eq:21or24}
\sigma(\kappa, x_3)\;:=\;
\begin{cases}
- \frac{C^{1\to 2}_0}{C^{1\to 2}_1},\;& \;\text{if only \eqref{eq:c202} holds}  \\
- \frac{C^{1\to 2}_1}{C^{1\to 2}_2},\;& \;\text{if only \eqref{eq:c001} holds}
\end{cases}.
\end{align}
Notice that $G$ is consistent. That means ${\mathcal K}\neq \emptyset$, i.e., there exists $\kappa\in {\mathbb R}^m_{>0}$ such that $G$ has positive steady states. So, by Lemma \ref{lm:gsys}, for any $\kappa^*\in {\mathcal K}$, 
$g_{1\to 2}(\kappa^*, x_2, x_3)=0$ indeed has positive solutions for $(x_2, x_3)$.  
Also notice that by \eqref{eq:A12}, for any $\kappa\in {\mathbb R}^m_{>0}$, $C^{1\to 2}_1(\kappa, x_3)=A^{1\to 2}_1(\kappa, x_3)=0$ has at most one positive solution for $x_3$. 
So, for any $\kappa^* \in {\mathcal K}$, 
there exists an open interval ${\mathcal I}\subset {\mathbb R}_{>0}$ such that for any $x_3\in {\mathcal I}$, $x_2=\sigma(\kappa^*, x_3)$ defined in \eqref{eq:21or24} is the unique nondegenerate positive solution of 
$g_{1\to 2}(\kappa^*, x_2, x_3)=0$, and for  any $x_3\in  \left({\mathbb R}_{>0}\backslash {\mathcal I}\right)$, $g_{1\to 2}(\kappa^*, x_2, x_3)=0$ has no positive solutions. Therefore, $(1,2,3)$ has the local  positive isoline property and $\sigma^*:=\sigma(\kappa^*,x_3):{\mathcal I}\to {\mathbb R}_{>0}$ gives the local  positive function.

{\it (6)}
If both \eqref{eq:c202} and \eqref{eq:c001} hold for $(i,j,k)=(\hat{i},\hat{j},\hat{k})$ and $(i,j,k)=(\hat{i},\hat{k},\hat{j})$, then
\[
a_{\hat{i}}^{(\hat{i})}=a_{\hat{j}}^{(\hat{i})}=a_{\hat{k}}^{(\hat{i})}=a_{\hat{i},\hat{j}}^{(\hat{i})}=a_{\hat{i},\hat{k}}^{(\hat{i})}=a_{\hat{j},\hat{k}}^{(\hat{i})}=0,
\]
That means $f_{\hat{i}}=a^{(\hat{i})}_0$, which contradicts to the hypothesis that $G$ is consistent and nondegenerate.

{\it (7)} Without loss of generality, we assume that $\left(\hat{i},\hat{j},\hat{k}\right)=(1,2,3)$.  If both of \eqref{eq:c201} and \eqref{eq:c004} hold  for $(i,j,k)=(1,2,3)$ and  both \eqref{eq:c202} and \eqref{eq:c001} hold for $(i,j,k)=(1,3,2)$, then by Remark \ref{rmk:czero}, we have    $g_{1\to 2}=C^{1\to 2}_1x_2$ and $g_{1\to 3}=C^{1\to 3}_1x_3$. And 
by \eqref{eq:c1}, we have 
\begin{align}
C^{1\to 2}_1 &=  a_{1,3}^{(1)}x_3
       \bigl(a_2^{(2)} + a_{2,3}^{(2)}x_3\bigr)
    -\bigl(a_1^{(2)} +a_{1,3}^{(2)}x_3\bigr)a_2^{(1)}
   \notag \\
    &= a_{1,3}^{(1)}a_{2,3}^{(2)}x_3^2+ \bigl(a_{1,3}^{(1)}a_2^{(2)} - a_{1,3}^{(2)}a_2^{(1)}\bigr)x_3-a_1^{(2)}a_2^{(1)} \label{eq:c1211}
\\
C^{1\to 3}_1 &=a_{1,3}^{(1)}\bigl(a_0^{(3)} + a_2^{(3)}x_2\bigr)-a_{1,3}^{(3)}\bigl(a_0^{(1)} + a_2^{(1)}x_2\bigr) \label{eq:c1311}
   \end{align}
   By Lemma \ref{lm:01} {\it (1)}, we have $a_{1,3}^{(1)}a_{2,3}^{(2)}\geq 0$ and $-a_1^{(2)}a_2^{(1)}\leq 0$. So, $C^{1\to 2}_1=0$ has at most one positive solution for $x_3$ for any $\kappa\in {\mathbb R}^m_{>0}$. Notice that $C^{1\to 3}_1$ is linear in $x_2$.  So,  we conclude that $g_{1\to 2}=0$ and $g_{1\to 3}=0$ have at most  one common positive solution $(x^*_2, x^*_3)$ for any  $\kappa \in {\mathbb R}^m_{>0}$, and hence, by Lemma \ref{lm:gsys}, $G$ admits at most one positive steady state. 

{\it (8)}  Without loss of generality, we assume that $\left(\hat{i},\hat{j},\hat{k}\right)=(1,2,3)$.  Notice that  by Remark \ref{rmk:czero}, if \eqref{eq:c201} or  \eqref{eq:c004} holds  for $(i,j,k)=(1,2,3)$ and  both \eqref{eq:c202} and \eqref{eq:c001} hold for $(i,j,k)=(1,3,2)$, then 
$C^{1\to 2}_1$ and $C^{1\to 3}_1$ are also given by \eqref{eq:c1211}--\eqref{eq:c1311}. Again, we have $g_{1\to 3}=C^{1\to 3}_1x_3$.

If  only \eqref{eq:c201} holds  for $(i,j,k)=(1,2,3)$, then  $g_{1\to 2}=C^{1\to 2}_1x_2+C^{1\to 2}_0$ with $C^{1\to 2}_0<0$.
   So, we can solve $x_2$ from the equation $g_{1\to 2}(\kappa, x_2, x_3)=0$ and get
    $$x_2\;=\;- \frac{C^{1\to 2}_0}{C^{1\to 2}_1}.$$
    If $a_{1,3}^{(1)}a_{2,3}^{(2)}a_1^{(2)}a_2^{(1)}$ is not the zero polynomial, then by \eqref{eq:c1211},  $C^{1\to 2}_1=0$ always has two real solutions for $x_3$ with opposite signs for any $\kappa^*\in {\mathbb R}^m_{>0}$. Denoted by  $z_{+}(\kappa^*)$  the positive solution. 
    Notice that $a_{1,3}^{(1)}a_{2,3}^{(2)}> 0$.
    Then, for any $x_3\in {\mathcal I}:= (z_{+}(\kappa^*), +\infty)$, we have $C^{1\to 2}_1>0$ and hence, $\sigma^*(x_3):=- \frac{C^{1\to 2}_0}{C^{1\to 2}_1}|_{\kappa=\kappa^*}$ gives the only nondegenerate positive solution of $g_{1\to 2}(\kappa^*,x_2,x_3)=0$ for $x_2$. And for any $x_3\in {\mathbb R}_{>0}\backslash{\mathcal I}$, $g_{1\to 2}(\kappa^*,x_2,x_3)=0$ has no positive solutions for $x_2$.
    Therefore, $(1,2,3)$ has the local positive isoline property. Notice that if  $a_{1,3}^{(1)}$ is the zero polynomial, then $C^{1\to 2}_1\leq 0$, which leads that
    $g_{1\to 2}=0$ has no positive solution for $x_2$. By Lemma \ref{lm:gsys}, this case cannot happen since $G$ is consistent. If $a_{2,3}^{(2)}a_1^{(2)}a_2^{(1)}$ is the zero polynomial, then by \eqref{eq:c1211}, $C^{1\to 2}_1$ becomes linear in $x_3$, and  we can always find an open interval ${\mathcal I}\subset {\mathbb R}_{>0}$ such that $(1,2,3)$ has the local positive isoline property in a way similar to the proof of {\it (5)}.

    If  only \eqref{eq:c004} holds  for $(i,j,k)=(1,2,3)$, then $g_{1\to 2}=x_2(C^{1\to 2}_2x_2+C^{1\to 2}_1)$ with $C^{1\to 2}_2>0$. 
    So, we can solve $x_2$ from the equation $g_{1\to 2}(\kappa, x_2, x_3)=0$ and get
    $$x_2\;=\;- \frac{C^{1\to 2}_1}{C^{1\to 2}_2}.$$  
    Then, the conclusion holds in a similar way as in the previous argument.

{\it (9)} By Remark \ref{rmk:czero}, one of \eqref{eq:c201} and \eqref{eq:c004} holds  for all  $(i,j,k)\in \mathfrak{S}$ is equivalent to the following three groups of conditions \eqref{eq:c0041}--\eqref{eq:c0043} hold simultaneously:

\begin{subequations}\label{eq:c0041}
\begin{align}
a^{(1)}_{1,2}=a^{(2)}_{1,2}=0, \label{eq:c0041a}\\
\multicolumn{1}{c}{\text{or}} \notag\\
a^{(1)}_0=a^{(1)}_{3}=a^{(2)}_0=a^{(2)}_{3}=0; \label{eq:c0041b}
\end{align}
\end{subequations}

\begin{subequations}\label{eq:c0042}
\begin{align}
a^{(1)}_{1,3}=a^{(3)}_{1,3}=0, \label{eq:c0042a}\\
\multicolumn{1}{c}{\text{or}} \notag\\
a^{(1)}_0=a^{(1)}_{2}=a^{(3)}_0=a^{(3)}_{2}=0; \label{eq:c0042b}
\end{align}
\end{subequations}

\begin{subequations}\label{eq:c0043}
\begin{align}
a^{(2)}_{2,3}=a^{(3)}_{2,3}=0, \label{eq:c0043a}\\
\multicolumn{1}{c}{\text{or}} \notag\\
a^{(2)}_0=a^{(2)}_{1}=a^{(3)}_0=a^{(3)}_{1}=0. \label{eq:c0043b}
\end{align}
\end{subequations}

\begin{table}[h]
\centering
\begin{tabular}{|c|c|c|}
\hline
  & $C^{1\to 2}_1$  & signs of $x^2_3$, $x_3$, $1$  \\
\hline
\eqref{eq:c0041a}, \eqref{eq:c0042a},  \eqref{eq:c0043a}  &   $a_{1}^{(1)}a_{2}^{(2)}- \bigl(a_1^{(2)} + a_{1,3}^{(2)}x_3\bigr)
       \bigl(a_2^{(1)} + a_{2,3}^{(1)}x_3\bigr)$&  $-$, $-$, $?$  \\
\hline
\eqref{eq:c0041b}, \eqref{eq:c0042b},  \eqref{eq:c0043b}  &   $\bigl(a_1^{(1)} + a_{1,3}^{(1)}x_3\bigr)
       \bigl(a_2^{(2)} + a_{2,3}^{(2)}x_3\bigr)-a^{(2)}_{1,3}a^{(1)}_{2,3}x^2_3$&  $?$, $+$, $+$  \\
\hline
 \eqref{eq:c0041a}, \eqref{eq:c0042b},  \eqref{eq:c0043b}  & $\bigl(a_1^{(1)} + a_{1,3}^{(1)}x_3\bigr)
       \bigl(a_2^{(2)} + a_{2,3}^{(2)}x_3\bigr)-a^{(2)}_{1,3}a^{(1)}_{2,3}x^2_3$   &  $?$, $+$, $+$ \\
\hline
\eqref{eq:c0041a}, \eqref{eq:c0042a},  \eqref{eq:c0043b}   & $a_1^{(1)} \bigl(a_2^{(2)} + a_{2,3}^{(2)}x_3\bigr)-a^{(2)}_{1,3}x_3\bigl(a_2^{(1)}+a^{(1)}_{2,3}x_3\bigr)$   &  $-$, $?$, $+$ \\
\hline
\end{tabular}
\caption{How conditions \eqref{eq:c0041}--\eqref{eq:c0043} affect $C^{1\to 2}_1$}
\label{tab:}
\end{table}
Notice that if \eqref{eq:c0041a} holds, then  $C^{1\to 2}_0 \equiv0$ 
and $g_{1\to 2}=x_2(C^{1\to 2}_2x_2+C^{1\to 2}_1)$. 
If \eqref{eq:c0041b} holds, then
 $C^{1\to 2}_2 \equiv0$ 
and $g_{1\to 2}=C^{1\to 2}_1x_2+C^{1\to 2}_0$. 
Table~\ref{tab:} lists the expression for $C_1^{1\to 2}$ under all possible combinations of conditions \eqref{eq:c0041}--\eqref{eq:c0043}. Note that $C_1^{1\to 2}$ is a polynomial in $x_3$ of degree at most 2. For each case, we also record the sign pattern of the coefficients of $x_3^2$, $x_3$, and the constant term, which will be useful for determining the existence of positive real roots of $C_1^{1\to 2}$.
So, for each case listed in Table \ref{tab:},  we can conclude that  for any $\kappa\in {\mathbb R}^m_{>0}$, $C^{1\to 2}_1=0$ has at most one positive solution for $x_3$. By symmetry,  for any $\kappa\in {\mathbb R}^m_{>0}$, $C^{1\to 3}_1=0$ has at most one positive solution for $x_2$. Then, we can prove the conclusion by a similar way to the proof of {\it (8)}. 

 {\it (10)} 
 Suppose $(i,j,k)\in \mathfrak{S}_3$. We begin with $(i,j,k)=(1,2,3)$.
It is known that one of the conditions \eqref{eq:c202}--\eqref{eq:c001} holds for $i=1$, $j=2$ and $k=3$.
By {\it (5)}, if only one of \eqref{eq:c202} and \eqref{eq:c001} holds, then $(1,2,3)$ has the local positive isoline property. 

Assume that both \eqref{eq:c202} and \eqref{eq:c001} hold, or one of \eqref{eq:c201} and \eqref{eq:c004} holds, for $(i,j,k)=(1,2,3)$. At this stage, for $(i,j,k)=(1,3,2)$, we also need only consider the cases where both \eqref{eq:c202} and \eqref{eq:c001} hold, or one of \eqref{eq:c201} and \eqref{eq:c004} holds; for the remaining cases, the conclusion {\it (5)} implies that $(1,3,2)$ has the local positive isoline property. By {\it (6)}, it is impossible that both \eqref{eq:c202} and \eqref{eq:c001} hold simultaneously for $(i,j,k)=(1,2,3)$ and $(i,j,k)=(1,3,2)$.  

By {\it (7)}, if  both of \eqref{eq:c201} and \eqref{eq:c004} hold for $(i,j,k)=(1,2,3)$ while both \eqref{eq:c202} and \eqref{eq:c001} hold for $(i,j,k)=(1,3,2)$, 
then $G$ admits at most one positive steady state. If only one of \eqref{eq:c201} and \eqref{eq:c004} holds for $(i,j,k)=(1,2,3)$ while both \eqref{eq:c202} and \eqref{eq:c001} hold for $(i,j,k)=(1,3,2)$, then by  {\it (8)}, 
$(1,2,3)$ has the local positive isoline property. 
By symmetry, if  one of \eqref{eq:c201} and \eqref{eq:c004} holds for $(i,j,k)=(1,3,2)$ while both \eqref{eq:c202} and \eqref{eq:c001} hold for $(i,j,k)=(1,2,3)$, then either $G$ admits at most one positive steady state, or $(1,3,2)$ has the local positive isoline property.
The remaining case is that one of \eqref{eq:c201} and \eqref{eq:c004} holds for both $(i,j,k)=(1,2,3)$ and $(i,j,k)=(1,3,2)$. By symmetry, we only need to consider the situation where one of \eqref{eq:c201} and \eqref{eq:c004} holds for all $(i,j,k)\in \mathfrak{S}_3$. The conclusion then follows from {\it (9)}. 
\end{proof}


{\it Proof of Theorem \ref{thm:nomss}} 
Notice that $G$ has at least $2$ nondegenerate positive
steady states for the given $\kappa\in {\mathbb R}^m_{>0}$. That means $G$ is consistent and nondegenerate.  
If there exists $(i,j,k)\in \mathfrak{S}_3$ such that  $C^{i\to j}_0C^{i\to j}_2$ is not the zero polynomial, then by Lemma \ref{lm:c0c2nonzero}, $(i,j,k)$ has the positive isoline property, and the conclusion follows from Lemma \ref{lm:hsys}. If for all $(i,j,k)\in \mathfrak{S}_3$, 
$C^{i\to j}_0C^{i\to j}_2$ is the zero polynomial, 
then by Remark \ref{rmk:czero}, one of the conditions \eqref{eq:c204}--\eqref{eq:c002} holds for all $(i,j,k)\in \mathfrak{S}_3$.
By Lemma \ref{lm:c0c2zero} {\it (1)} and {\it (3)}, 
the condition \eqref{eq:c204}  cannot hold, and
the condition \eqref{eq:c003} cannot hold. 
By Lemma \ref{lm:c0c2zero} {\it (2)} and {\it (4)}, if  \eqref{eq:c203} holds or \eqref{eq:c002} holds for some $(i,j,k)\in \mathfrak{S}_3$, then $(i,j,k)$ has the positive isoline property. Otherwise, 
 one of \eqref{eq:c202}--\eqref{eq:c001} holds for all $(i,j,k)\in \mathfrak{S}_3$, then by Lemma \ref{lm:c0c2zero} {\it (10)}, there exists  
 $(\hat{i},\hat{j},\hat{k})$ has the local positive isoline property and the conclusion follows from Lemma \ref{lm:hsys}.    \hfill$\square$

\section{Methods}\label{sec:methods}
Our computational pipeline for classifying multistable $(3,6,3)$ quadratic zero-one networks is summarized in Fig.~\ref{fig:pipeline}. After a preprocessing step that enumerates and reduces the family to inequivalent representatives (using the algorithm of \cite{jiao2026equivalence}), the detection stage proceeds through three core steps: Jacobian-sign screening to eliminate networks incapable of multiple steady states, real root classification to identify networks admitting $3$ to $5$ nondegenerate positive steady states, and stability verification via Jacobian determinant signs to certify multistability. This pipeline refines and extends the multistability-detection strategy of \cite{jiao2025multistability} in three key aspects. First, the injectivity-based screening in \cite{jiao2025multistability} is replaced by Jacobian-sign screening, which by Theorem~\ref{thm:nomss} is  necessary  for excluding networks incapable of multistationarity (here, we do not need to check the boundary steady states or the dissipativity like the degree-theory-based methods \cite{conradi2017identifying}). Second, Theorem~\ref{thm:nomss} guarantees that multistability requires at least $3$ nondegenerate positive steady states, while Remark~\ref{rmk:bezout} gives an upper bound of $5$ via the BKK mixed volume; together these confine the real root classification step to the precise range of $3$ to $5$ positive steady states. Third, Lemma~\ref{lm:stab} provides a necessary and sufficient stability criterion for nondegenerate steady states in three-dimensional zero-one networks, improving upon the necessary  condition given in \cite[Lemma~38]{jiao2025multistability}.

\begin{center}
\captionsetup{type=figure}
\begin{tikzpicture}[
  node distance=0.45cm,
  block/.style={text width=6cm, align=center, font=\footnotesize}
]
\node (n1) [inp, block]
{Enumerate and reduce $(3,6,3)$ quadratic zero-one networks};

\node (n6) [pro, block, below=of n1]
{Jacobian-sign screening};

\node (n7) [pro, block, below=of n6]
{Real root classification:\\ $3$--$5$ nondegenerate positive steady states};

\node (n8) [pro, block, below=of n7]
{Stability verification via\\ Jacobian determinants};

\node (n9) [inp, block, below=of n8]
{$373$ multistable networks};

\draw [arrow] (n1) -- (n6);
\draw [arrow] (n6) -- (n7);
\draw [arrow] (n7) -- (n8);
\draw [arrow] (n8) -- (n9);
\end{tikzpicture}
\caption{The computational pipeline.}
\label{fig:pipeline}
\end{center}

We now describe in detail how each step of the pipeline is implemented, illustrating the procedure with concrete computational examples.
 
First, using a standard approach, for each network we form the transformed Jacobian matrix $\operatorname{J}_f(p,\lambda):={\mathcal N}{\rm diag}(\sum^{n}_{\ell=1}\lambda_{\ell}L^{(\ell)}){\mathcal Y}^{\top}{\rm diag}(p)$, where $L^{(1)},\ldots,L^{(n)}$ denote the generators of the flux cone $\{v\in {\mathbb R}_{>0}^m|{\mathcal N}v=0\}$.  
Then, we apply a Jacobian-sign
screening: if $\det\!\left(\operatorname{J}_f(p,\lambda)\right)$ is sign-definite on the positive
orthant, then by Theorem \ref{thm:nomss}, the network admits at most one positive steady state and hence cannot
be multistable, so we discard it.

Second, on the surviving networks we run \texttt{\seqsplit{RealRootClassification}}
\cite{xia2016automated} to decide whether there is a parameter region in which
the network has several nondegenerate positive steady states. 
By Theorem \ref{thm:nomss}, 
for $(3,m,3)$ quadratic zero-one networks, multistability requires at least $3$ nondegenerate positive
steady states, while  Remark \ref{rmk:bezout} limits
$5$ positive steady states at most; we therefore
search for between $3$ and $5$  nondegenerate positive steady states.
The following Example \ref{ex:rrc-exclusion} illustrates how
\texttt{\seqsplit{RealRootClassification}} excludes a network
at the second step.
\begin{example}
\label{ex:rrc-exclusion}
Consider the $(3,6,3)$ quadratic zero-one network
\[
\begin{array}{@{}l@{\qquad}l@{}}
X_2+X_3 \xrightarrow{\kappa_1} X_1+X_2,
&
X_1 \xrightarrow{\kappa_2} X_1+X_3,
\\[2pt]
X_1 \xrightarrow{\kappa_3} X_1+X_2+X_3,
&
X_1 \xrightarrow{\kappa_4} X_2,
\\[2pt]
X_1+X_3 \xrightarrow{\kappa_5} X_1,
&
X_1+X_2 \xrightarrow{\kappa_6} X_3.
\end{array}
\]
Its steady-state system is
\[
\begin{aligned}
f_1={}&-\kappa_4x_1-\kappa_6x_1x_2
              +\kappa_1x_2x_3,\\
f_2={}&\kappa_3x_1+\kappa_4x_1
              -\kappa_6x_1x_2,\\
f_3={}&\kappa_2x_1+\kappa_3x_1+\kappa_6x_1x_2
              -\kappa_5x_1x_3-\kappa_1x_2x_3.
\end{aligned}
\]
To determine whether this system can have at least $3$
nondegenerate positive steady states, we apply the following
Maple command:
\begin{lstlisting}[style=maplestyle]
F := [-k4*x1-k6*x1*x2+k1*x2*x3, k3*x1+k4*x1-k6*x1*x2,
      k2*x1+k3*x1+k6*x1*x2-k5*x1*x3-k1*x2*x3]:
P := [x1,x2,x3,k1,k2,k3,k4,k5,k6]:
R := PolynomialRing(P):
rrc := RealRootClassification(F, [], P, [], 6, 3..5, R);
\end{lstlisting}

The output contains

\begin{lstlisting}[style=mapleoutput]
RealRootClassification: FINAL RESULT:
RealRootClassification: There is no given number of real solution(s)!
\end{lstlisting}
 The output certifies that there
is no choice of positive rate constants for which the
steady-state system has between $3$ and $5$ nondegenerate
positive solutions. 
By Remark~\ref{rmk:bezout}, the system cannot
have more than $5$ nondegenerate positive steady states.
Hence, it admits at most two nondegenerate positive steady
states. Therefore, by Theorem~\ref{thm:nomss}, this network cannot
admit multistability and is discarded before the stability
verification step.
\end{example}

Third, if a network can have at least $3$ nondegenerate positive steady states,  for each sample rate-constant point \(\kappa^*\) returned by
\texttt{\seqsplit{RealRootClassification}}, we isolate the
nondegenerate positive steady states
\(
x^{(1)},\ldots,x^{(N)}
\). 
If \(N=4\) or \(N=5\), then Theorem~\ref{thm:nomss} guarantees
that at least two of these positive steady states are stable.
Hence, the corresponding parameter point $\kappa^*$ yields multistability.
It remains to consider the case \(N=3\). We evaluate
\(
\det\!\left(
\operatorname{Jac}_f(\kappa^*,x^{(\ell)})
\right)
\) for all $\ell$.
By  
Lemma~\ref{lm:stab}, $x^{(\ell)}$ is
stable if and only if \(
\det\!\left(
\operatorname{Jac}_f(\kappa^*,x^{(\ell)})
\right)
\) is negative.  Remark that if we cannot find any sample point $\kappa^*$ such that the multistability is exhibited, then we conclude that the network has no multistability for any $\kappa \in {\mathbb R}^m_{>0}$, as \texttt{RealRootClassification} returns at least one sample point from each open connected region of the complement of the discriminant hypersurface  of $f$ \cite{chen2013semialgebraic}.
 The following example illustrates how
\texttt{RealRootClassification} determines the number
of nondegenerate positive steady states and how the subsequent
Jacobian-determinant test verifies multistability.

\begin{example}
\label{ex:multistable-network}
We revisit the cell-fate decision network presented in Example \ref{ex:pip}. To determine whether this network can admit at least $3$
positive steady states, we apply the same
\texttt{\seqsplit{RealRootClassification}} command used in
the preceding example:

\begin{lstlisting}[style=maplestyle]
F := [k2*x2-k6*x1*x2-k5*x1*x3+k3*x2*x3, k1-k6*x1*x2-k3*x2*x3, k4*x1+k2*x2-k5*x1*x3-k3*x2*x3]:
P := [x1,x2,x3,k1,k2,k3,k4,k5,k6]:
R := PolynomialRing(P):
rrc := RealRootClassification(F, [], P, [], 6, 3..5, R);
\end{lstlisting}

In contrast to Example~\ref{ex:rrc-exclusion}, the output is
affirmative and has the following form:

\begin{lstlisting}[style=mapleoutput]
RealRootClassification: FINAL RESULT:
RealRootClassification: The system has given number of
real solution(s) IF AND ONLY IF
RealRootClassification: [0 <= R[1], R[2] < 0].
\end{lstlisting}

Here, \(R[1]\) and \(R[2]\) are explicit polynomials in the
rate constants. Their expressions are omitted because they
are lengthy. This output confirms the existence of a
nonempty positive parameter region in which the system has
between $3$ and $5$ nondegenerate positive steady states.
To check whether $4$ or $5$ positive steady states can
actually occur, we further run
\begin{lstlisting}[style=maplestyle]
res45 := RealRootClassification(F, N, P, H, 6, 4..5, R);
\end{lstlisting}
Maple returns
\begin{lstlisting}[style=mapleoutput]
RealRootClassification: FINAL RESULT:
RealRootClassification: There is no given number of real solution(s)!
\end{lstlisting}
provided that the corresponding border polynomial that generates the discriminant hypersurface is nonzero.
Thus, outside the discriminant hypersurface, the system cannot
have $4$ or $5$ positive steady states.  That means we cannot have more than $3$ nondegenerate positive steady states. Hence, the number of nondegenerate positive steady states is at most $3$.

We then run the sample-point version for exactly $3$ nondegenerate positive
steady states:
\begin{lstlisting}[style=maplestyle]
res3 := RealRootClassification(F, N, P, H, 6, 3, R, output = samples);
\end{lstlisting}
The procedure returned a single sample point (this suggests that the parameter region supporting $3$ nondegenerate positive steady states is connected):
\begin{lstlisting}[style=mapleoutput]
[k6 = 1/2, k5 = 1/2, k4 = 1/2, k3 = 1/2,k2 =31/4096, k1 = 119/1024]
\end{lstlisting}
Therefore, for the stability verification, we use the following
rational parameter point:
\[
\kappa^*=(\kappa_1,\kappa_2,\kappa_3,\kappa_4,\kappa_5,\kappa_6)
=
\left(
\frac{119}{1024},
\frac{31}{4096},
\frac12,
\frac12,
\frac12,
\frac12
\right).
\]
After substituting these values into the steady-state system
and isolating its positive solutions, we obtain exactly $3$
positive steady states:
\[
\begin{aligned}
x^{(1)}
&=(0.04224916,\;3.33414617,\;0.02746041)^\top,\\
x^{(2)}
&=(0.10764284,\;1.10613000,\;0.10247883)^\top,\\
x^{(3)}
&=(0.19301892,\;0.46942690,\;0.30209943)^\top.
\end{aligned}
\]
The Jacobian matrix of the system is
\[
\operatorname{Jac}_f(\kappa,x)=
\begin{pmatrix}
-\kappa_6x_2-\kappa_5x_3
&
\kappa_2-\kappa_6x_1+\kappa_3x_3
&
-\kappa_5x_1+\kappa_3x_2
\\
-\kappa_6x_2
&
-\kappa_6x_1-\kappa_3x_3
&
-\kappa_3x_2
\\
\kappa_4-\kappa_5x_3
&
\kappa_2-\kappa_3x_3
&
-\kappa_5x_1-\kappa_3x_2
\end{pmatrix}.
\]
 It is directly to check that
\[
\det\!\left(
\operatorname{Jac}_f(\kappa^*,x^{(1)})
\right)<0,\;
\det\!\left(
\operatorname{Jac}_f(\kappa^*,x^{(2)})
\right)>0,\;
\det\!\left(
\operatorname{Jac}_f(\kappa^*,x^{(3)})
\right)<0.
\]
It follows from Lemma~\ref{lm:stab} that \(x^{(1)}\) and
\(x^{(3)}\) are stable, whereas \(x^{(2)}\) is unstable.
Therefore, the network admits multistability; see
Fig.~\ref{example_multistability}.

\begin{center}
\captionsetup{type=figure}
\includegraphics[width=0.95\columnwidth]{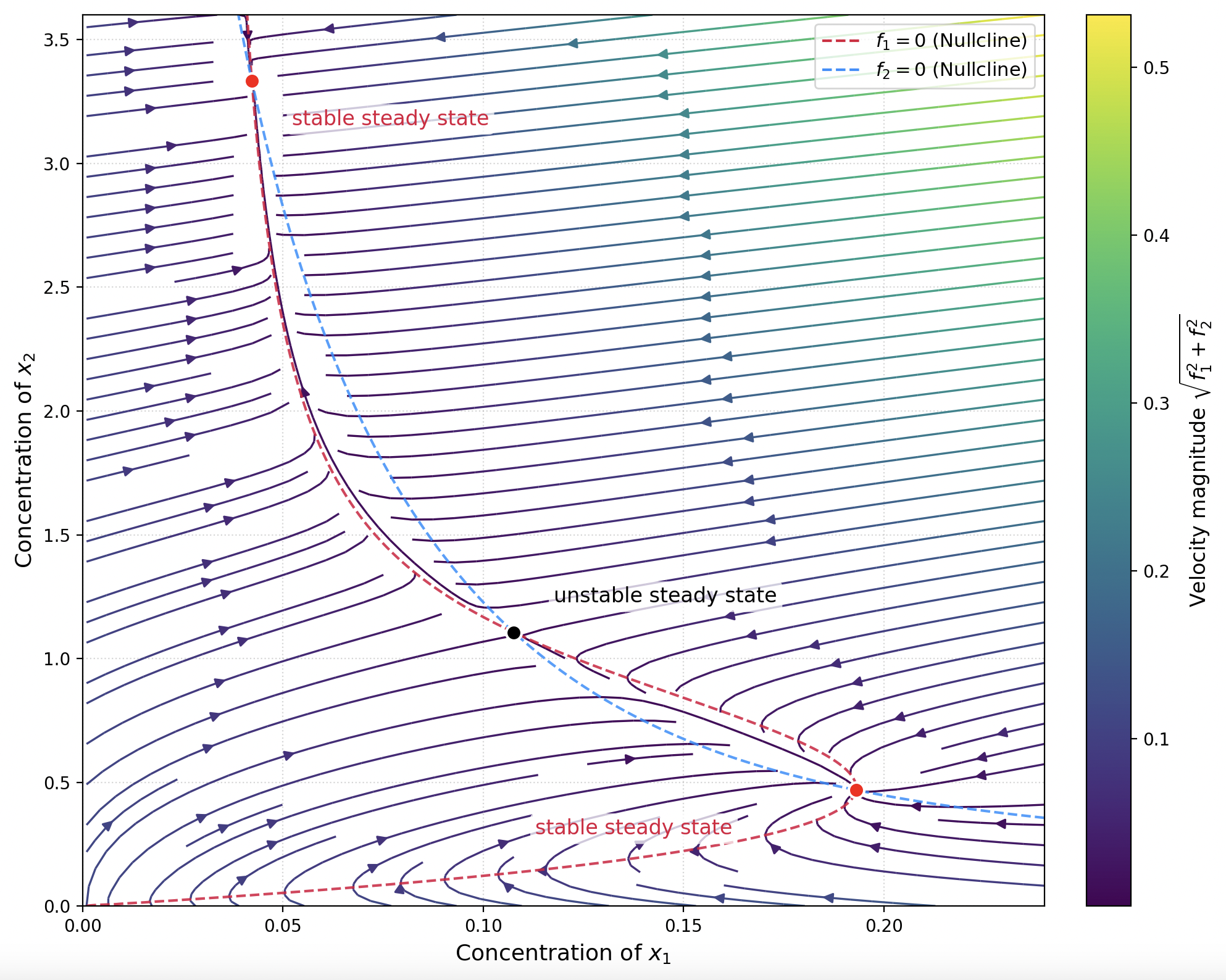}
\caption{The $3$ positive steady states of the network in Example~1. The
states $x^{(1)}$ and $x^{(3)}$ are stable, while $x^{(2)}$ is unstable.}
\label{example_multistability}
\end{center}

 As shown in Fig.~\ref{example_multistability}, the cell-fate decision network exhibits bistability:
two stable steady states correspond to distinct cell fates. 
 Recall that $x_1, x_2$ and $x_3$ denote the concentrations of $\mathbf{D}$, $\mathbf{P}$ and $\mathbf{C}$. 
The state with high
$\mathbf{P}$ and low $\mathbf{D}$ ($x^{(1)}$) represents the pluripotent
state, where the stemness factor dominates; the state with low $\mathbf{P}$ and
moderate $\mathbf{D}$ ($x^{(3)}$) represents the differentiated state.
The intermediate state ($x^{(2)}$) is unstable and acts as a threshold
separating the two basins of attraction. This switch-like behavior enables
robust cell-fate commitment without graded input signals.
\end{example}


\section{Results}
\label{sec:results}
 
We apply the pipeline of Section~\ref{sec:methods} to the class of
\((3,6,3)\) quadratic zero-one networks. Large-scale grouping and comparison
are implemented in \texttt{C}, while 
 steady-state-system generation, Gr\"obner basis computations (needed in the equivalence reduction \cite{jiao2025equivalence}), and Jacobian-sign screening are carried out
in \texttt{Mathematica} \cite{mathematica2024}. We use
\texttt{Maple} \cite{maplesoft2020} for
\texttt{RealRootClassification} and \texttt{Python} scripts to parse and
summarize the output. All computations were performed on a \(2.70\)\,GHz
Intel Core i5-11400H processor with \(16\)\,GB of memory under Windows~11.
The full computational results are available
online.\footnote{\url{https://github.com/zjdong-sudo/paper-multistability-results}}

Following \cite{jiao2026equivalence}, we exclude the strictly quadratic subclass (where every reaction has a molecularity of exactly~
$2$). Since it was shown in \cite{jiao2026equivalence} that this subclass admits no positive steady states for generic rate constants (as the homogeneous steady-state system is always degenerate), we restrict our search to the remaining quadratic networks. Table~\ref{tab:equiv_results}
reports the number of objects surviving each filter of the equivalence
reduction according to the steps of \cite[Algorithm 1]{jiao2026equivalence}. 
 
\begin{center}
\captionsetup{type=table}
\caption{Equivalence reduction for $(3,6,3)$ quadratic zero-one networks}
\label{tab:equiv_results}
\small
\setlength{\tabcolsep}{6pt}
\begin{tabular*}{\columnwidth}{@{\extracolsep{\fill}}lcc@{}}
\toprule
Filter & Objects retained & Time \\
\midrule
\multicolumn{3}{c}{Stoichiometric matrices} \\
\midrule
Candidate matrices (rank $3$) & 168699 & 7 min \\
Consistency filtering         & 87624  & 17 min \\
Same-form reduction           & 14895  & 40 min \\
\midrule
\multicolumn{3}{c}{Reaction networks} \\
\midrule
Reactant-matrix reconstruction & 530480 & 20 min \\
RREF / permutation grouping       & 409232 & 20 min \\
Vacuous/monomial-ideal removal    & 349763 & 1 h \\
Steady-state ideal reduction   & 349597 & 30 min \\
\bottomrule
\end{tabular*}
\end{center}
 
The reduction removes most redundancy before the multistability detection stage. Notably, the
steady-state ideal reduction lowers the count only from $349763$ to $349597$,
so for larger searches this final step may be omitted when a complete
equivalence classification is not required.
 
We then apply the multistability detection stage to the surviving networks.
Table~\ref{tab:multi_results} reports the result of its three filters according to Fig.~\ref{fig:pipeline}.
 
\begin{center}
\captionsetup{type=table}
\caption{Multistability detection for $(3,6,3)$ quadratic zero-one networks}
\label{tab:multi_results}
\small
\setlength{\tabcolsep}{6pt}
\begin{tabular*}{\columnwidth}{@{\extracolsep{\fill}}lcc@{}}
\toprule
Filter & Count & Time \\
\midrule
Jacobian-sign screening                       & 245468 & 1 min \\
Real root classification                      & 375 & 4 h \\
Stability verification            & 373    & 1 min \\
\bottomrule
\end{tabular*}
\end{center}
 
Notice that \texttt{RealRootClassification} returns $375$ networks (from the $245468$ networks that passed the Jacobian-sign screening) that
admit at least $3$ nondegenerate positive steady states; remarkably, each of them admits
exactly $3$. Finally, the stability verification confirms that \(373\)  networks are identified as multistable.

\begin{theorem}\label{thm:main}
A \((3,6,3)\) quadratic zero-one reaction network admits at most
$3$ nondegenerate positive steady states. There are \(375\) such
networks that admit exactly $3$ nondegenerate positive steady states, among
which \(373\) admit multistability.
\end{theorem}

\begin{example}\label{ex:special}
The following two networks are the only two among the \(375\) candidates that were not certified as multistable by the subsequent stability test.
The first network is
\[
\begin{array}{@{}l@{\qquad}l@{}}
X_3 \xrightarrow{\kappa_1} X_1,
&
X_2 \xrightarrow{\kappa_2} X_1+X_3,
\\[2pt]
X_2+X_3 \xrightarrow{\kappa_3} 0,
&
X_1 \xrightarrow{\kappa_4} 0,
\\[2pt]
X_1+X_3 \xrightarrow{\kappa_5} X_1+X_2+X_3,
&
X_1+X_2 \xrightarrow{\kappa_6} X_1+X_2+X_3.
\end{array}
\]
The second network is
\[
\begin{array}{@{}l@{\qquad}l@{}}
X_3 \xrightarrow{\kappa_1} X_1+X_2+X_3,
&
X_2 \xrightarrow{\kappa_2} 0,
\\[2pt]
X_2+X_3 \xrightarrow{\kappa_3} X_1,
&
X_1 \xrightarrow{\kappa_4} 0,
\\[2pt]
X_1+X_3 \xrightarrow{\kappa_5} X_1+X_2,
&
X_1+X_2 \xrightarrow{\kappa_6} X_1+X_2+X_3.
\end{array}
\]
As an example, we illustrate the stability verification for the second network. Its
steady-state system is
\[
\begin{aligned}
f_1
&=-\kappa_4x_1+\kappa_1x_3+\kappa_3x_2x_3,\\
f_2
&=-\kappa_2x_2+\kappa_1x_3+\kappa_5x_1x_3
-\kappa_3x_2x_3,\\
f_3
&=\kappa_6x_1x_2-\kappa_5x_1x_3-\kappa_3x_2x_3.
\end{aligned}
\]
Notice that \texttt{RealRootClassification} returns only one sample point for $3$ nondegenerate positive solutions:
\[
\kappa^*
=
\left(
\frac{555}{2048},
\frac{103}{65536},
\frac12,\frac12,\frac12,\frac12
\right).
\]
And the $3$ nondegenerate positive steady states
are \[
\begin{aligned}
x^{(1)}&=(0.03170,\ 0.5191,\ 0.02987)^\top,\\
x^{(2)}&=(0.09948,\ 0.6188,\ 0.08570)^\top,\\
x^{(3)}&=(0.2491,\ 0.7782,\ 0.1887)^\top.
\end{aligned}
\]
The Jacobian matrix of the mass-action system is
\[
\operatorname{Jac}_f(\kappa,x)
=
\begin{pmatrix}
-\kappa_4
&
\kappa_3x_3
&
\kappa_1+\kappa_3x_2
\\
\kappa_5x_3
&
-\kappa_2-\kappa_3x_3
&
\kappa_1+\kappa_5x_1-\kappa_3x_2
\\
\kappa_6x_2-\kappa_5x_3
&
\kappa_6x_1-\kappa_3x_3
&
-\kappa_5x_1-\kappa_3x_2
\end{pmatrix}.
\]
 It is directly to check that
\[
\det\!\left(
\operatorname{Jac}_f(\kappa^*,x^{(1)})
\right)>0,\;
\det\!\left(
\operatorname{Jac}_f(\kappa^*,x^{(2)})
\right)<0,\;
\det\!\left(
\operatorname{Jac}_f(\kappa^*,x^{(3)})
\right)>0.
\]
So, it follows from Lemma~\ref{lm:stab} that \(x^{(2)}\) is stable,
whereas \(x^{(1)}\) and \(x^{(3)}\) are unstable. Therefore, this
sample parameter point has $3$ nondegenerate positive steady
states but does not yield multistability; see
Fig.~\ref{remaining-network}.

This network lacks the defining motif of cell-fate switches: a self-sustaining positive feedback for the pluripotency factor ${\mathbf P}$ ($X_2$). Instead,  ${\mathbf P}$ ($X_2$)  is merely produced and degraded without autocatalytic maintenance, while ${\mathbf D}$ ($X_1$)  acts catalytically rather than as a terminal effector. Such architecture does not match known gene-regulatory logic for stem-cell maintenance, making the observed ``three steady states with only one stable" behavior biologically implausible as a decision mechanism.

\begin{center}
\captionsetup{type=figure}
\includegraphics[width=0.95\columnwidth]{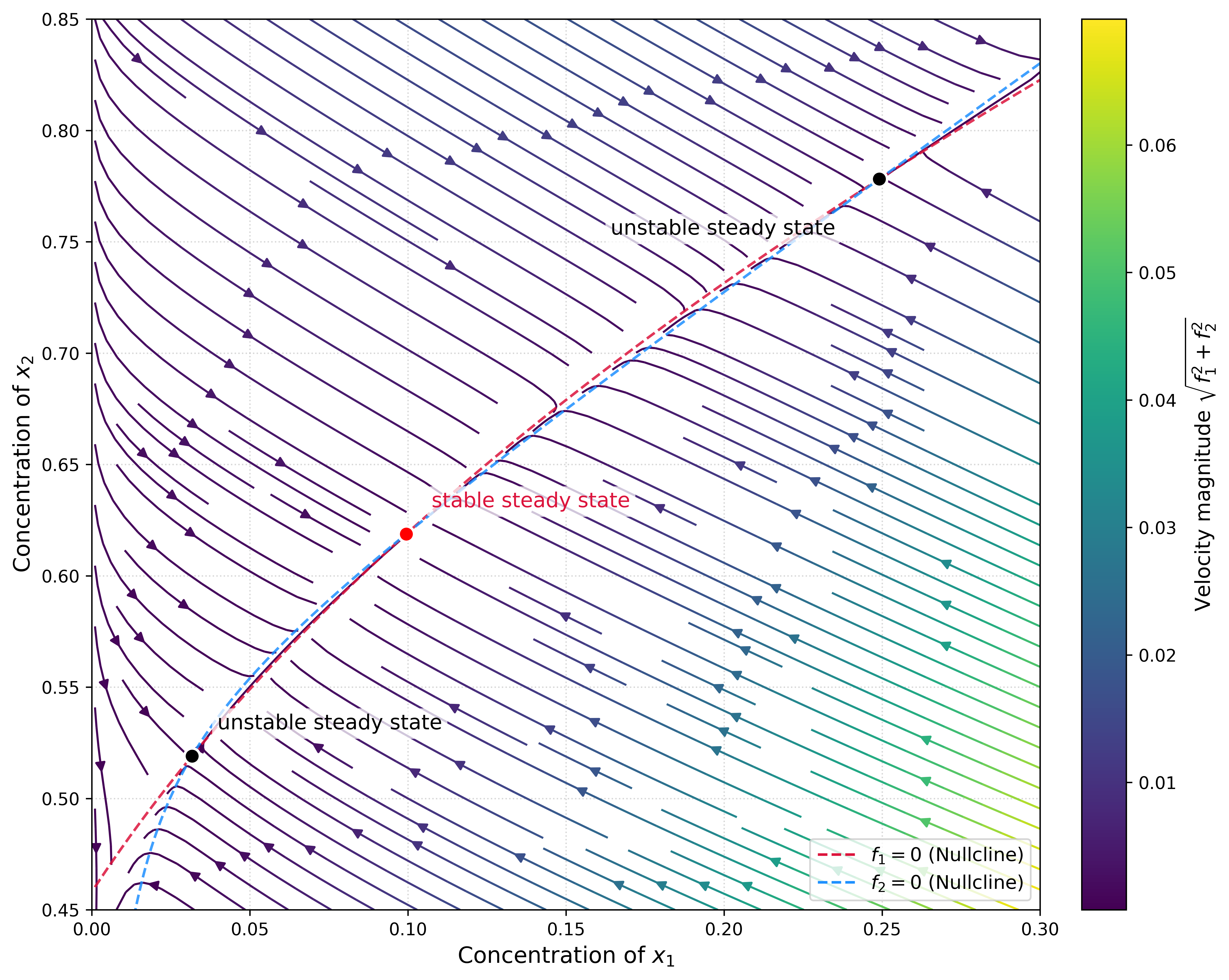}
\caption{The $3$ positive steady states of the second network
in Example~\ref{ex:special} at the parameter point
\(\kappa^*\). The state \(x^{(2)}\) is stable, while
\(x^{(1)}\) and \(x^{(3)}\) are unstable.}
\label{remaining-network}
\end{center}


\end{example}

\subsection{Why exactly three positive steady states?}\label{sec:why-three}

A striking feature of Table~\ref{tab:multi_results} is that, although
\texttt{\seqsplit{RealRootClassification}} was queried for parameter regions admitting
between $3$ and $5$ positive steady states, every one of the $375$ networks
that passed this filter admits exactly $3$; none attains $4$ or
more. Although a complete explanation of this bound is left for future work, here we
provide an explanation for Example 
\ref{ex:multistable-network}.

\begin{example}\label{ex:why-three}
We revisit the steady-state system in
Example~\ref{ex:multistable-network}. Let
\[
I=\langle f_1,f_2,f_3\rangle
\subseteq
\mathbb{Q}(\kappa_1,\ldots,\kappa_6)[x_1,x_2,x_3]
\]
be its steady-state ideal. We compute a Gr\"obner basis of
\(I\) with respect to the lexicographic order
\(
x_1\succ x_3\succ x_2.
\)
By the elimination theorem, the resulting Gr\"obner basis
contains a polynomial involving only \(x_2\). Up to a
nonzero constant factor, this elimination polynomial is
\[
\begin{aligned}
Q(\kappa, x_2)={}&
9\kappa_2\kappa_3\kappa_6^2x_2^4
+ \left(
6\kappa_2\kappa_3\kappa_4\kappa_6
-3\kappa_1\kappa_3\kappa_6^2
\right)x_2^3
+
\left(
\kappa_2\kappa_3\kappa_4^2
+2\kappa_1\kappa_3\kappa_4\kappa_6
\right)x_2^2\\
&+
\left(
\kappa_1\kappa_3\kappa_4^2
-2\kappa_1^2\kappa_5\kappa_6
\right)x_2-
2\kappa_1^2\kappa_4\kappa_5.
\end{aligned}
\]
The steady-state system also 
gives
\[
x_1=
\frac{2\kappa_1}{\kappa_4+3\kappa_6x_2},
\qquad
x_3=
\frac{\kappa_1(\kappa_4+\kappa_6x_2)}
{\kappa_3x_2(\kappa_4+3\kappa_6x_2)}.
\]
These two quantities are positive and uniquely determined
whenever \((\kappa, x_2)\in {\mathbb R}^{m+1}_{>0}\). Hence, the positive steady states are in
one-to-one correspondence with the positive real roots of
\(Q(\kappa^*, x_2)\) for any given $\kappa^* \in {\mathbb R}^m_{>0}$.

Since \(Q(\kappa^*,x_2)\) has at most $4$ real roots and the product of all roots is 
$-\frac{2{\kappa_1^*}^2\kappa^*_4\kappa^*_5}{9\kappa^*_2\kappa^*_3{\kappa_6^*}^2}<0$, 
\(Q(\kappa^*,x_2)\) has at most $3$ positive real roots. 
Therefore,
the network in Example~\ref{ex:multistable-network} admits
at most $3$ positive steady states.
\end{example}
\section{Discussion}\label{sec:discussion}

In this work, we have provided a complete classification of multistability in $(3, 6, 3)$ quadratic zero-one networks, identifying $373$ networks that exhibit bistability. These minimal networks serve as a comprehensive ``atlas" of the simplest biochemical switches capable of cell-fate decisions.

A natural and significant extension of this work is the systematic analysis of the structural commonalities among these $373$ multistable networks. While our results establish the existence and limits of multistability in this class, understanding the design principles underlying these networks remains an open challenge. Future work should investigate whether these networks universally share specific topological motifs such as particular configurations of positive feedback loops or mutual inhibition and how the roles of individual species (e.g., as substrates or enzymes) are distributed across this ensemble.

Furthermore, it is crucial to theoretically elucidate why these minimal networks are strictly bounded by three positive steady states. Establishing a rigorous proof for this upper bound would reveal the intrinsic algebraic constraints governing small-scale reaction systems and distinguish the capacity limits of minimal switches from larger, more complex networks. 

Uncovering these structural signatures and theoretical bounds would bridge the gap between the exhaustive enumeration presented here and the rational design of robust synthetic biological circuits. To facilitate such studies, we have made the full list of these $373$ networks publicly available in our repository.

\section*{Declarations}

Conflict of interest: The authors declare that they have no conflict of interest.

Data availability: All the supporting codes and data are available online at
\href{https://github.com/zjdong-sudo/paper-multistability-results}
{https://github.com/zjdong-sudo/paper-multistability-results}.
The repository contains three main folders, ``code'', ``data'', and ``results'',
together with a file ``README.md''.
The folder ``code'' contains the implementations of the equivalence-reduction
procedure 
and the
multistability-detection procedure presented in
Section~\ref{sec:methods}.
The folder ``data'' contains the intermediate computational data corresponding
to the results reported in Tables~\ref{tab:equiv_results} and
\ref{tab:multi_results}.
The folder ``results'' contains the \(375\) networks admitting $3$
nondegenerate positive steady states, together with the stability-verification
results identifying \(373\) of them as multistable, as summarized in
Theorem~\ref{thm:main}.



\bibliography{reference}

@inproceedings{jiao2025equivalence,
  author    = {Jiao, Yue and Tang, Xiaoxian},
  title     = {An efficient algorithm for determining the equivalence of zero-one reaction networks},
  booktitle = {ISSAC '25: Proceedings of the 2025 International Symposium on Symbolic and Algebraic Computation},
  pages     = {277--283},
  year      = {2025},
  doi       = {10.1145/3747199.3747571}
}

@article{jiao2025multistability,
  author={Jiao, Yue and Tang, Xiaoxian and Zeng, Xiaowei},
  title={Multistability of small zero-one reaction networks},
  journal={Journal of Mathematical Biology},
  volume={91},
  number={6},
  pages={82},
  year={2025},
  doi={10.1007/s00285-025-02306-w}
}

@article{banaji2024oscillations,
  author={Banaji, Murad and Boros, Bal{\'a}zs and Hofbauer, Josef},
  title={Oscillations in three-reaction quadratic mass-action systems},
  journal={Studies in Applied Mathematics},
  volume={152},
  number={1},
  pages={249--278},
  year={2024},
  doi={10.1111/sapm.12639}
}

@article{tang2023hopf,
  author={Tang, Xiaoxian and Wang, Kaizhang},
  title={Hopf bifurcations of reaction networks with zero-one stoichiometric coefficients},
  journal={SIAM Journal on Applied Dynamical Systems},
  volume={22},
  number={3},
  pages={2459--2489},
  year={2023},
  doi={10.1137/22M1519754}
}

@incollection{conradi2019multistationarity,
  author    = {Conradi, Carsten and Pantea, Casian},
  title     = {Multistationarity in biochemical networks: results, analysis, and examples},
  booktitle = {Algebraic and Combinatorial Computational Biology},
  pages     = {279--317},
  year      = {2019},
  doi       = {10.1016/B978-0-12-814066-6.00009-X}
}

@article{banaji2016injectivity,
  author  = {Banaji, Murad and Pantea, Casian},
  title   = {Some results on injectivity and multistationarity in chemical reaction networks},
  journal = {SIAM Journal on Applied Dynamical Systems},
  volume  = {15},
  number  = {2},
  pages   = {807--869},
  year    = {2016},
  doi     = {10.1137/15M1034441}
}

@article{craciun2005multiple1,
  author  = {Craciun, Gheorghe and Feinberg, Martin},
  title   = {Multiple equilibria in complex chemical reaction networks: {I}. The injectivity property},
  journal = {SIAM Journal on Applied Mathematics},
  volume  = {65},
  pages   = {1526--1546},
  year    = {2005},
  doi     = {10.1137/S0036139904440278}
}

@article{conradi2017identifying,
  author  = {Conradi, Carsten and Feliu, Elisenda and Mincheva, Maya and Wiuf, Carsten},
  title   = {Identifying parameter regions for multistationarity},
  journal = {PLoS Computational Biology},
  volume  = {13},
  number  = {10},
  pages   = {e1005751},
  year    = {2017},
  doi     = {10.1371/journal.pcbi.1005751}
}

@article{craciun2008homotopy,
  author  = {Craciun, Gheorghe and Helton, J. William and Williams, Raymond J.},
  title   = {Homotopy methods for counting reaction network equilibria},
  journal = {Mathematical Biosciences},
  volume  = {216},
  number  = {2},
  pages   = {140--149},
  year    = {2008},
  doi     = {10.1016/j.mbs.2008.09.001}
}

@article{enciso2014fixed,
  author  = {Enciso, German A.},
  title   = {Fixed points and convergence in monotone systems under positive or negative feedback},
  journal = {International Journal of Control},
  volume  = {87},
  number  = {2},
  pages   = {301--311},
  year    = {2014},
  doi     = {10.1080/00207179.2013.830336}
}

@misc{maplesoft2020,
  author       = {{Maplesoft, a division of Waterloo Maple Inc.}},
  title        = {Maplesoft},
  year         = {2020},
  address      = {Waterloo, Ontario}
}

@article{banaji2018inheritance,
  author  = {Banaji, Murad and Pantea, Casian},
  title   = {The inheritance of nondegenerate multistationarity in chemical reaction networks},
  journal = {SIAM Journal on Applied Mathematics},
  volume  = {78},
  number  = {2},
  pages   = {1105--1130},
  year    = {2018},
  doi     = {10.1137/16M1103506}
}

@misc{mathematica2024,
  author       = {{Wolfram Research, Inc.}},
  title        = {Mathematica},
  year         = {2024},
  note         = {Version 14.2},
  address      = {Champaign}
}

@article{banaji2018oscillation,
  author  = {Banaji, Murad},
  title   = {Inheritance of oscillation in chemical reaction networks},
  journal = {Applied Mathematics and Computation},
  volume  = {325},
  pages   = {191--209},
  year    = {2018},
  doi     = {10.1016/j.amc.2017.12.012}
}

@article{hilioti2008oscillatory,
  author  = {Hilioti, Zoe and Sabbagh, Walid Jr. and Paliwal, Swati and Bergmann, Andreas and Goncalves, M. D. and Bardwell, Lee and others},
  title   = {Oscillatory phosphorylation of yeast Fus3 {MAP} kinase controls periodic gene expression and morphogenesis},
  journal = {Current Biology},
  volume  = {18},
  number  = {21},
  pages   = {1700--1706},
  year    = {2008},
  doi     = {10.1016/j.cub.2008.09.027}
}

@article{sible2007mathematical,
  author  = {Sible, Jill C. and Tyson, John J.},
  title   = {Mathematical modeling as a tool for investigating cell cycle control networks},
  journal = {Methods},
  volume  = {41},
  number  = {2},
  pages   = {238--247},
  year    = {2007},
  doi     = {10.1016/j.meth.2006.08.003}
}

@article{tyson2022timekeeping,
  author  = {Tyson, John J. and Novak, B{\'e}la},
  title   = {Time-keeping and decision-making in the cell cycle},
  journal = {Interface Focus},
  volume  = {12},
  number  = {4},
  pages   = {20210075},
  year    = {2022},
  doi     = {10.1098/rsfs.2021.0075}
}

@article{janiakspens2005kinetic,
  author  = {Janiak-Spens, Frances and Cook, Paul F. and West, Alan H.},
  title   = {Kinetic analysis of {YPD1}-dependent phosphotransfer reactions in the yeast osmoregulatory phosphorelay system},
  journal = {Biochemistry},
  volume  = {44},
  number  = {1},
  pages   = {377--386},
  year    = {2005},
  doi     = {10.1021/bi048433s}
}

@article{ferrell1998allornone,
  author  = {Ferrell, James E. Jr. and Machleder, E. M.},
  title   = {The biochemical basis of an all-or-none cell fate switch in {Xenopus} oocytes},
  journal = {Science},
  volume  = {280},
  number  = {5365},
  pages   = {895--898},
  year    = {1998},
  doi     = {10.1126/science.280.5365.895}
}

@article{bagowski2001bistability,
  author  = {Bagowski, Christopher and Ferrell, James E. Jr.},
  title   = {Bistability in the {JNK} cascade},
  journal = {Current Biology},
  volume  = {11},
  pages   = {1176--1182},
  year    = {2001},
  doi     = {10.1016/S0960-9822(01)00330-X}
}

@article{xiong2003memory,
  author  = {Xiong, Wei and Ferrell, James E. Jr.},
  title   = {A positive-feedback-based bistable `memory module' that governs a cell fate decision},
  journal = {Nature},
  volume  = {426},
  number  = {6965},
  pages   = {460--465},
  year    = {2003},
  doi     = {10.1038/nature02089}
}

@article{craciun2006understanding,
  author  = {Craciun, Gheorghe and Tang, Yangzhong and Feinberg, Martin},
  title   = {Understanding bistability in complex enzyme-driven reaction networks},
  journal = {Proceedings of the National Academy of Sciences of the United States of America},
  volume  = {103},
  number  = {23},
  pages   = {8697--8702},
  year    = {2006},
  doi     = {10.1073/pnas.0602767103}
}

@article{chen2013semialgebraic,
  author  = {Chen, Changbo and Davenport, James H. and Moreno Maza, Marc and Xia, Bican and Xiao, Rong},
  title   = {Computing with semi-algebraic sets: relaxation techniques and effective boundaries},
  journal = {Journal of Symbolic Computation},
  volume  = {52},
  pages   = {72--96},
  year    = {2013},
  doi     = {10.1016/j.jsc.2012.05.013}
}

@book{xia2016automated,
  author    = {Xia, Bican and Yang, Lu},
  title     = {Automated inequality proving and discovering},
  publisher = {World Scientific},
  address   = {Singapore},
  year      = {2016},
  doi       = {10.1142/9951}
}

@misc{jiao2026equivalence,
  author        = {Jiao, Yue and Tang, Xiaoxian},
  title         = {Determining the equivalence of small zero-one reaction networks},
  year          = {2026},
  eprint        = {2503.00008},
  archiveprefix = {arXiv},
  primaryclass  = {q-bio.MN},
  url           = {https://arxiv.org/abs/2503.00008}
}

@article{bernstein1975numberofroots,
  author  = {Bernstein, D. N.},
  title   = {The number of roots of a system of equations},
  journal = {Functional Analysis and its Applications},
  volume  = {9},
  number  = {3},
  pages   = {183--185},
  year    = {1975},
  doi     = {10.1007/BF01075595}
}

@article{TorresFeliu2021,
  author  = {Torres, Ang{\'e}lica and Feliu, Elisenda},
  title   = {Symbolic proof of bistability in reaction networks},
  journal = {SIAM Journal on Applied Dynamical Systems},
  volume  = {20},
  number  = {1},
  pages   = {1--37},
  year    = {2021},
  doi     = {10.1137/20M1326672}
}

@article{rouillier1999solving,
   author={Rouillier, Fabrice},
  title={Solving zero-dimensional systems through the rational univariate representation},
 journal={Applicable Algebra in Engineering, Communication and Computing},
  volume={9},
  number={5},
  pages={433--461},
  year={1999},
  doi = {10.1007/s002000050114}
}

@inproceedings{gianni1989algebraic,
 author={Gianni, Patrizia and Mora, Teo},
  title={Algebraic solution of systems of polynomial equations using {Gr\"obner} bases},
   booktitle={Applied Algebra, Algebraic Algorithms and Error-Correcting Codes},
  series={LNCS},
  volume={356},
  pages={247--257},
  year={1989},
  doi={10.1007/3-540-51082-6_83},
  publisher={Springer},
  address		= {Heidelberg}
}

\end{document}